\documentclass[leqno,english,11pt]{amsart}
\usepackage[T1]{fontenc}
\usepackage{lmodern,amsmath,amssymb,amsthm,mathtools,microtype}
\usepackage[margin=1in]{geometry}
\usepackage[colorlinks=true,linkcolor=blue,citecolor=blue,urlcolor=blue]{hyperref}
\numberwithin{equation}{section}
\newtheorem{theorem}{Theorem}[section]
\newtheorem{proposition}[theorem]{Proposition}
\newtheorem{lemma}[theorem]{Lemma}
\newtheorem{corollary}[theorem]{Corollary}
\theoremstyle{definition}
\newtheorem{definition}[theorem]{Definition}
\newtheorem{assumption}[theorem]{Assumption}
\theoremstyle{remark}
\newtheorem{remark}[theorem]{Remark}
\newcommand{\R}{\mathbb R}

\newcommand{\dd}{\,\mathrm d}
\newcommand{\diver}{\operatorname{div}}

\newcommand{\norm}[1]{\left\|#1\right\|}
\newcommand{\abs}[1]{\left|#1\right|}

\allowdisplaybreaks[2]

\makeatletter
\renewcommand\subsection{\@startsection{subsection}{2}{\z@}%
  {-3.25ex\@plus -1ex \@minus -.2ex}
  {1.5ex \@plus .2ex}
  {\normalfont\large\bfseries}}        
\makeatother

\begin{document}
\title[Joint limits in tissue growth]{Joint inviscid, incompressible, and continuous phenotype limit in nonlocal models of tissue growth}
\author[C.-C. Lai]{Chen-Chih Lai}
\address{Department of Mathematics, Columbia University, New York, NY 10027, USA}
\curraddr{Department of Mathematical Sciences,
Middle Tennessee State University,
Murfreesboro, TN 37132, USA}
\email{clai@mtsu.edu}
\author[H. Lu]{Hongbo Lu}
\address{Department of Applied Mathematics and Statistics, Stony Brook University, Stony Brook, NY 11794, USA}
\email{hongbo.lu@stonybrook.edu}
\date{}
\begin{abstract}
We consider a multispecies Brinkman tissue-growth system proposed by Debiec, Mandal, and Schmidtchen \cite{DMS-JDE2025}. This system adopts a power pressure law determined by the total spatial density and models the effect of phenotype-dependent growth rates on the density dynamics of cells with different phenotypes. We examine the simultaneous vanishing-viscosity, stiff-pressure, and continuous-phenotype limit of this system. For families of bounded weak solutions with suitable initial conditions, every sequence converges, up to a subsequence, to a continuous-phenotype incompressible Darcy system, which presents a novel biological model. This result not only addresses connections between Brinkman tissue-growth systems with discrete phenotypes and those with continuous phenotypes, but also demonstrates the robustness of such connections under both the incompressible and inviscid limits.
\end{abstract}
\maketitle
\tableofcontents

\section{Introduction and main result}
Mechanical models of tissue growth combine pressure-driven transport with local proliferation; see, for example, \cite{Tang-CAMB2013,Perthame2015}. In a multispecies description, phenotypes share a velocity field but may have different responses to pressure. We consider three simultaneous changes of scale: Brinkman transport approaches Darcy transport, the pressure law approaches a hard congestion constraint, and a finite collection of phenotypes approaches a continuous trait distribution.

The energy-evolution approach of David, Jacobs, and Kim \cite{DJK-arXiv2025} provides a framework for joint Brinkman-to-Darcy and incompressible limits. The continuous-phenotype reconstruction used by D\k{e}biec, Mandal, and Schmidtchen \cite{DMS-JDE2025} treats finite-to-continuous phenotype limits and their compatibility with vanishing viscosity. Related inviscid and stiff-pressure limits are studied in \cite{DDMS-SIMA2024,ES-JMFM2025,Perthame2014,Perthame2015}. Here the issue is the compatibility of the three limits along an arbitrary simultaneous sequence. For the system governed by power pressure law and strictly pressure-decreasing growth proposed in \cite{DMS-JDE2025}, we give a proof which retains the full Brinkman dissipation and explicitly handles phenotype-dependent reactions under joint limit including stiff pressure limit.

The broader literature includes nonlocal approximations of anisotropic cross-diffusion systems \cite{DS-NA2025}. For stiff-pressure limits, Hele--Shaw models and their relation to porous-medium dynamics have been studied in \cite{Perthame2014,Mellet-JFA2017,KimPozar-TAMS2018,Kim2018}, with active-motion effects considered in \cite{PQTV-arXiv2014}. Congestion also appears in crowd-transport models \cite{AKY-IOP2014,Maury2010} and in incompressible limits of viscous-flow models \cite{Vauchelet-arXiv2017}. For viscosity-solution methods form another part of this analytical background, a general reference is \cite{CIL-BAMS1992}.

Phenotype structure connects the present problem to models of cancer-cell heterogeneity \cite{Lorenzi2016-BiologyDirect,VillaChaplainLorenzi-SJAM2021}, interfaces between populations with different mobilities \cite{LLP-KRM2017}, and selection--mutation dynamics with several fitness peaks \cite{LorenziPouchol-Nonlinearity2020}. Related free-boundary limits incorporate nutrients \cite{David-JDE2021} and convective effects \cite{David-CPAM2024}.

\subsection{Model and simultaneous sequence}\mbox{}
Fix $T>0$, set $Q_T=\R^d\times(0,T)$ and $I=(0,1)$, and let:
\begin{equation}\label{seq}
 N_j\longrightarrow\infty,\qquad k_j\longrightarrow\infty,\qquad \nu_j\longrightarrow0.
\end{equation}
After discarding finitely many indices, assume $k_j\ge3$ and $0<\nu_j\le1$.
Set $a_{j,i}=i/N_j$ and $I_{j,i}=((i-1)/N_j,i/N_j]$. We consider
\begin{equation}\label{model}
\begin{split}
 \partial_t n_j^i&=\diver(n_j^i\nabla W_j)+n_j^iG(p_j,a_{j,i}),\\
 \rho_j&=\frac1{N_j}\sum_{i=1}^{N_j}n_j^i,
 \qquad p_j=P_j(\rho_j),\\
 P_j(r)&=\frac{k_j}{k_j-1}r^{k_j-1},
 \qquad W_j-\nu_j\Delta W_j=p_j.
\end{split}
\end{equation}
The densities $n_j^i$ are normalized phenotype densities; the physical contribution of species $i$ to the total density is $n_j^i/N_j$, and $\rho_j$ is the total spatial density.
Define the exact reconstructions for $a\in I_{j,i}$:
\begin{equation}\label{reconstruction}
 u_j(x,t,a)=n_j^i(x,t),\qquad
 G_j(q,a)=G(q,a_{j,i}).
\end{equation}
In particular,
\begin{equation}\label{total-eq}
 \rho_j=\int_Iu_j\dd a,\qquad
 R_j=\int_Iu_jG_j(p_j,a)\dd a,\qquad
 \partial_t\rho_j=\diver(\rho_j\nabla W_j)+R_j.
\end{equation}
Write $A_j=R_j/\rho_j$ on $\{\rho_j>0\}$ and $A_j=0$ otherwise. Then $R_j=\rho_j A_j$. All subsequent subsequences are extracted from the single sequence \eqref{seq}.

\subsection{Hypotheses and solution class}
\begin{assumption}\label{ass:growth}
The function $G\in C^1([0,\infty)\times[0,1])$ satisfies $G(0,a)>0$ and
\begin{equation}\label{strict-growth}
 \partial_qG(q,a)\le-c<0\qquad(q\ge0,\ a\in[0,1]).
\end{equation}
Choose $P_H>0$ such that $G(P_H,a)\le0$ for all $a$. The zeros of $G(\cdot,a)$ need not coincide.
\end{assumption}
\begin{assumption}\label{ass:data}
There exists a fixed, compactly supported nonnegative profile:
\[
 n_0(x,a)\in L^\infty_c(\mathbb R^d\times[0,1]), \text{with} \int_0^1 n_0(x,a) da \leq 1 \text{ for all } x
\]
independent of $j$, and a constant $L_{\rm in}<\infty$ such that,
for almost every $x\in\mathbb R^d$,
\[
 |n_0(x,a)-n_0(x,b)|
 \le L_{\rm in}|a-b|
 \qquad\text{for all }a,b\in[0,1].
\]
We use this Lipschitz representative in the phenotype variable
to prescribe the initial species densities:
\[
 n_j^{i,\rm in}(x)=n_0(x,i/N_j),
 \qquad i=1,\ldots,N_j,
 \qquad
 \rho_j^{\rm in}(x)
 =\frac1{N_j}\sum_{i=1}^{N_j}n_j^{i,\rm in}(x).
\]
There are constants $P_0,M_0$, independent of $j,i$,
such that
\begin{align}
 &P_j(\rho_j^{\rm in})\le P_0
       \quad\text{a.e.},
       \label{initial-pressure}\\
 &\int_{\mathbb R^d}(1+|x|)
       \rho_j^{\rm in}(x)\,dx\le M_0,
       \label{initial-moment}
\end{align}
\end{assumption}
These are assumptions on the entire approximating family. In particular, \eqref{initial-pressure} guarantees a uniform $L^\infty$ bound on $\rho_j^{\rm in}$. Though the moment bound in \eqref{initial-moment} follows from compactly supported assumption on $n_0$, denoting the actual upper bound as $M_0$ here simplifies later arguments.

\begin{definition}\label{def:solution}
For a fixed $j$, a bounded nonnegative weak solution of \eqref{model} has nonnegative densities satisfying
\[
 n_j^i\in C([0,T];L^1(\R^d))\cap L^\infty(Q_T),
\]
with the prescribed strong $L^1$ initial trace. It satisfies $p_j=P_j(\rho_j)$ and $W_j=(1-\nu_j\Delta)^{-1}p_j$, where the inverse is the positive convolution resolvent on $\R^d$, and obeys, for every $\varphi\in C_c^\infty(\R^d\times[0,T))$,
\begin{equation}\label{weak-species}
 -\iint_{Q_T}n_j^i\partial_t\varphi
 +\iint_{Q_T}n_j^i\nabla W_j\cdot\nabla\varphi
 =\iint_{Q_T}\varphi n_j^iG(p_j,a_{j,i})
 +\int_{\R^d}n_j^{i,\rm in}\varphi(x,0)\dd x.
\end{equation}
We use standard Sobolev-space and weak-formulation conventions as in \cite{Evans2010}.
\end{definition}

Then we introduce the main theorem for this paper:

\begin{theorem}\label{thm:main}
Under Assumptions \ref{ass:growth}--\ref{ass:data}, let $(n_j^i,p_j,W_j)$ be any family of solutions in Definition \ref{def:solution}. There is a subsequence and nonnegative functions $n(x,t,a)$ and $p(x,t)$, with $\rho=\int_I n\dd a$, such that
\begin{align}
 u_j&\rightharpoonup n&&\text{in }L^2(Q_T\times I),\label{convergence-density}\\
 \rho_j&\rightharpoonup \rho&&\text{in }L^2(Q_T),\notag\\
 p_j,W_j&\to p&&\text{in }L^2(Q_T),\label{convergence-pressure}\\
 \nabla W_j&\to \nabla p&&\text{in }L^2(Q_T),\label{convergence-velocity}\\
 \nu_j|\Delta W_j|^2&\to 0&&\text{in }L^1(Q_T).\label{convergence-defect}
\end{align}
Here $n\in L^\infty(0,T;L^1(\R^d\times I)\cap L^2(\R^d\times I))$ and $p\in L^\infty(Q_T)\cap L^2(0,T;H^1(\R^d))$. The limit satisfies
\begin{equation}\label{limit-system}
\begin{cases}
 \partial_t n=\diver(n\nabla p)+nG(p,a),\\
 \partial_t\rho=\diver(\rho \nabla p)+\displaystyle\int_I nG(p,a)\dd a,\\
 \diver(p\nabla p)-|\nabla_xp|^2+p\int_{I}nG(p,a)\dd a=0,\\
 0\le\rho\le1,\quad p\ge0,\quad p(1-\rho)=0,
\end{cases}
\end{equation}
with initial data $n^{\rm in}$ and $\rho^{\rm in}$. The first equation holds in distributions on $Q_T\times I$ and, equivalently, for almost every phenotype $a$. The other two equations for $\rho$ and $p$ holds in distributions on $Q_T$. Also $\rho\nabla p=\nabla p$ almost everywhere, so the equation for $\rho$ can also be written as:
$$\partial_t\rho=\Delta p+\displaystyle\int_I nG(p,a)\dd a.$$
\end{theorem}

\subsection{Idea of the proof}
Our proof strategy is inspired by \cite{DJK-arXiv2025}. The main difficulty of directly adapting their results here are: 
\begin{enumerate}
    \item The phenotype growth difference may create extreme differences in phenotype distribution, and disallow bounding $\norm{n}_{L^{\infty}}$ by comparing it against $\rho$.
    \item The $L^2$ bounds needed for weak convergence of individual phenotype $n(\cdot,\cdot,a)$ for each $a$ cannot be easily obtained as $\norm{n}_{L^{\infty}}$ estimate is impossible.
    \item Identification for the limiting pressure $p$ is not entirely clear under continuous phenotype limit and stiff pressure limit.
\end{enumerate}
Below we summarize our proof which addresses these difficulties.

We first establish uniform bounds by direct estimates and extract a subsequence with weak limits $n,\rho,p$, and the weak limit of the reaction term denoted as $R$, whose expression in terms of $n$ and $p$ is not yet known. Then, we give a weak-product argument to establish the congestion constraints $0\leq\rho\leq1$, $p\geq0$, and $p(1-\rho)=0$. Then we use a separate truncated-pressure estimate to pass to the total-density equation $\partial_t\rho=\Delta p+R$ through weak convergence, before establishing any strong convergence results. 

This equation and the congestion constraints then yield the limiting pressure energy identity. Next, the strict decrease of growth in pressure controls the reaction term in the approximate energy identity. Comparing the approximate and limiting identities proves strong convergence of $\nabla W_j$ and shows that $\nu_j\iint_{Q_T}|\Delta W_j|^2$ tends to zero. The same argument provides a weighted pressure estimate from which strong convergence of $p_j$ follows. These strong convergences finally allow us to identify $R=\int_I nG(p,a)\,da$ and pass to the phenotype equation, and identifying an equation for limiting pressure $p$.

\subsection{Future directions}
In Assumption \ref{ass:growth}, the growth term $G$ is assumed to be strictly decreasing. This assumption is used to supply a coercive pressure variance. It would be interesting to consider more general growth terms that are merely nonincreasing, for which a different compactness step would be needed. 

Closely related to our compactness argument is the question of stronger convergence for the individual species density. The phenotype density equation is established using joint weak compactness, and the limiting individual species density is only established as a weak limit. To establish strong convergence of the individual species density therefore would require a more sophisticated regularity argument, that is a new challenge.

Beyond these compactness and regularity challenges, another interesting direction is to consider a more general kernel $K$ than the screened Poisson kernel $(I-\nu\Delta)^{-1}$ in the velocity-potential equation $W_j-\nu\Delta W_j = p_j$ \cite{Perthame2015}, so that $W_j=K*p_j$. In this case, a different estimate strategy on $\nabla W$ is needed.

In addition to these analytical extensions, it is also interesting to incorporate further biological mechanisms into the model. Mutation provides a further biological motivation for phenotype-structured models; see \cite{Tomlinson1996}. Our strategy has potential extensions to models with phenotype transitions, such as kernel-based switching, environment-driven phenotype drift, and phenotype diffusion. We will present the precise assumptions, the new proof, and the resulting limit systems in the upcoming paper \cite{LL-phenotype-transition}.

\subsection{Structure of the paper}
In Section \ref{sec:estimates}, we obtain a-priori estimates and related weak convergence results.
In Section \ref{sec:weak-identification}, we derive equations and inequalities that the weak limits must follow.
In Section \ref{sec:reaction-with-bv}, we establish a growth limsup estimate that is crucial in showing $\nu_j \abs{\Delta W_j}^2$ goes to $0$ in $L^1(Q_T)$.
In Seciont \ref{sec:strong-identification}, we derive the reaction identity and complete the proof of the main theorem.

\section{A-priori estimates and the exact pressure dissipation}\label{sec:estimates}
Set
\[
 P=\max(P_0,P_H),\qquad
 B=\max_{[0,P]\times[0,1]}|G|,\qquad
 L_a=\max_{[0,P]\times[0,1]}|\partial_aG|,\qquad L_p=\max_{[0,P]\times[0,1]}|\partial_qG|.
\]
Constants below may depend on these quantities, $M_0,T,d$, but never on $j$ or $i$.

Since it is hard to establish regularity results for $n_j^i$ directly, yet the chain rules involving this function will be used frequently later, we propose such computations as a lemma here. 

\begin{lemma}[Chain rules]
\label{lem:weak-solution-chain-rules}
Fix $j$, and let $(n_j^i,p_j,W_j)$ be a bounded weak solution as in Definition~\ref{def:solution}.
Write $G_j^i=G(p_j,a_{j,i})$, then for every $\beta\in C^1(\mathbb R)$ with $\beta(0)=0$:
\begin{align}
 \partial_t\beta(n_j^i)
 &=
 \operatorname{div}\bigl(\beta(n_j^i)\nabla W_j\bigr)
 +\bigl[n_j^i\beta'(n_j^i)-\beta(n_j^i)\bigr]\Delta W_j
 +\beta'(n_j^i)n_j^iG_j^i,
 \label{eq:species-chain-rule}\\
 \partial_t\beta(\rho_j)
 &=
 \operatorname{div}\bigl(\beta(\rho_j)\nabla W_j\bigr)
 +\bigl[\rho_j\beta'(\rho_j)-\beta(\rho_j)\bigr]\Delta W_j
 +\beta'(\rho_j)R_j.
 \label{eq:total-chain-rule}
\end{align}
These identities hold in distributions, including the initial terms $\beta(n_j^{i,\rm in})$ and $\beta(\rho_j^{\rm in})$.

Moreover, setting $d_j^i=n_j^{i+1}-n_j^i$,
\begin{equation}\label{eq:species-difference-absolute-value}
 \partial_t|d_j^i|
 =
 \operatorname{div}\bigl(|d_j^i|\nabla W_j\bigr)
 +\operatorname{sgn}(d_j^i)
       \bigl(n_j^{i+1}G_j^{i+1}-n_j^iG_j^i\bigr),
\end{equation}
with initial datum $|d_j^{i,\rm in}|$ and
$\operatorname{sgn}(0)=0$.
\end{lemma}

All arguments are at fixed $j$; their preliminary constants need not be uniform in $j$. And as the proof is mostly about technical details not related to other contents in the paper, we put it in Appendix \ref{sec:lemma2.1-proof}.

\begin{proposition}[A-priori estimates]
\label{prop:bounds}
Under Assumptions~\ref{ass:growth}--\ref{ass:data}, let $(n_j^i,p_j,W_j)$ be a family of bounded weak solutions. Let $K\subset\mathbb R^d$ be a compact spatial support of $n_0$, and set:
\[
\begin{aligned}
 &P=\max(P_0,P_H),\qquad
 B=\max_{[0,P]\times[0,1]}|G|,\qquad
 C_\rho=\max\{1,P^{1/2}\},\\
 &C_p=\frac32\max\{1,P\},\qquad
 C_{\rm in}=\|n_0\|_\infty,\qquad
 C_*=C_{\rm in}+2L_{\rm in}.
\end{aligned}
\]

Then:
\begin{align}
 &0\le p_j,W_j\le P,\qquad
 0\le\rho_j\le
 b_j:=\left(\frac{k_j-1}{k_j}P\right)^{1/(k_j-1)}
 \le C_\rho,\qquad b_j\to1,
 \label{pressure-bounds}\\
 &\sup_{j,t}\int_{\mathbb R^d}(\rho_j+p_j+W_j)\,dx
 \le(1+2C_p)M_0e^{BT}.
 \label{mass-bounds}
\end{align}

The individual species satisfy:
\begin{align}
 &\|n_j^i(t)\|_{L^1(\mathbb R^d)}
 \le e^{Bt}\|n_j^{i,\rm in}\|_{L^1(\mathbb R^d)}
 \le C_{\rm in}|K|e^{Bt},
 \label{eq:individual-L1-bound}\\
 &\|n_j^i(t)\|_{L^2(\mathbb R^d)}^2
 \le C_\rho C_*|K|e^{3Bt}.
 \label{eq:species-quadratic-bound}
\end{align}

Moreover,
\begin{equation}\label{eq:trait-joint-quadratic-bound}
 \|u_j\|_{L^\infty(0,T;L^2(\mathbb R^d\times(0,1)))}^2
 \le C_\rho C_*|K|e^{3BT}.
\end{equation}

Consequently, $\rho_j,p_j,W_j$ are uniformly bounded in $L^q(Q_T)$ for every $q\in[1,\infty]$, and the species are uniformly bounded in $L^\infty(0,T;L^q(\mathbb R^d))$ for every $q\in[1,2]$. The timewise estimates hold almost everywhere in $t$, and the species mass estimate holds for every $t\in[0,T]$.
\end{proposition}

\begin{proof}

Fix $j$. The Brinkman equation gives:
\[
 W_j=\mathcal K_{\nu_j}*p_j,\qquad
 \mathcal K_{\nu_j}\ge0,\qquad
 \mathcal K_{\nu_j}(x)=\mathcal K_{\nu_j}(-x),\qquad
 \|\mathcal K_{\nu_j}\|_1=1.
\]

Define $A_j=R_j/\rho_j$ on positive density and zero on the set where $\rho_j=0$. On $\{\rho_j>0\}$,
\[
 A_j=\sum_{i=1}^{N_j}
 \frac{n_j^i}{N_j\rho_j}G(p_j,a_{j,i}),
 \qquad
 \frac{n_j^i}{N_j\rho_j}\ge0,\qquad
 \sum_{i=1}^{N_j}\frac{n_j^i}{N_j\rho_j}=1.
\]

Lemma~\ref{lem:weak-solution-chain-rules} gives:
\begin{equation}\label{pressure-equation}
 \partial_t p_j-\operatorname{div}(p_j\nabla W_j)
 =(k_j-2)p_j\Delta W_j+(k_j-1)p_jA_j.
\end{equation}

To explain the maximum principle, first consider smooth pressure and a spatial maximum $p_j(x_t,t)=M_j(t)$. Positivity and unit mass of the kernel imply:
\[
 W_j(x_t,t)
 =\int_{\mathbb R^d}\mathcal K_{\nu_j}(x_t-z)p_j(z,t)\,dz
 \le M_j(t),
 \qquad
 \Delta W_j(x_t,t)
 =\frac{W_j(x_t,t)-M_j(t)}{\nu_j}\le0.
\]

If $M_j(t)\ge P$, then $G(M_j(t),a)\le G(P,a)\le0$ for every phenotype. Thus $A_j(x_t,t)\le0$, and since $\nabla p_j(x_t,t)=0$, the pressure equation gives:
\[
 \partial_t p_j(x_t,t)
 =(k_j-1)M_j(t)
   \bigl(\Delta W_j(x_t,t)+A_j(x_t,t)\bigr)\le0.
\]

Hence $P$ is an upper barrier for the pressure.

For bounded weak solutions, this conclusion is justified
by the positive-part chain rule established in Lemma \ref{lem:weak-solution-chain-rules}.
Set $y_j=(p_j-P)_+$ and $H_j(r)=\mathbf1_{\{r>P\}}((k_j-2)r+P)$. Then:
\[
 \partial_t y_j-\operatorname{div}(y_j\nabla W_j)
 =H_j(p_j)\Delta W_j
  +(k_j-1)\mathbf1_{\{p_j>P\}}p_jA_j.
\]

The reaction term $\mathbf1_{\{p_j>P\}}p_jA_j$ is non-positive, while kernel symmetry and
monotonicity of $H_j$ give:
\[
\begin{aligned}
 \int_{\mathbb R^d}H_j(p_j)(W_j-p_j)\,dx =-\frac12
 \iint_{\mathbb R^d\times\mathbb R^d}
 \mathcal K_{\nu_j}(x-z)
 \bigl[H_j(p_j(x))-H_j(p_j(z))\bigr]
 \bigl[p_j(x)-p_j(z)\bigr]\,dx\,dz
 \le0.
\end{aligned}
\]

These integrations are finite because for individual $j$'s, we have $p_j\in L^1_x\cap L^\infty_x$ and $H_j(p_j)\le(k_j-1)p_j$. Consequently,
\[
 \|y_j(t)\|_{L^1(\mathbb{R}^d)}\le\|y_j(0)\|_{L^1(\mathbb{R}^d)}=0,
\]
which proves $0\le p_j\le P$ uniformly in $j$.

Young's convolution inequality now gives
\[
 \|W_j\|_{L^\infty(Q_T)}
 \le\|\mathcal K_{\nu_j}\|_{L^1(\mathbb{R}^d)}
       \|p_j\|_{L^\infty(Q_T)}
 \le P.
\]
Since the kernel and pressure are nonnegative, $W_j\ge0$.
Inverting the pressure law therefore yields
\[
 0\le\rho_j
 =\left(\frac{k_j-1}{k_j}p_j\right)^{1/(k_j-1)}
 \le b_j
 :=\left(\frac{k_j-1}{k_j}P\right)^{1/(k_j-1)}
 \le C_\rho,
 \qquad b_j\to1.
\]
Finally, the definition of $B$ gives $|G(p_j,a)|\le B$, and
\[
 |R_j|
 \le\frac1{N_j}\sum_{i=1}^{N_j}
       n_j^i|G(p_j,a_{j,i})|
 \le B\rho_j,
 \qquad |A_j|\le B,
\]
where the last bound also holds on the set where $\rho_j = 0$ by the convention $A_j=0$ on that set.

Integrating the total-density and species equations and applying Gronwall inequality gives:
\[
 \|\rho_j(t)\|_{L^1(\mathbb{R}^d)}\le M_0e^{Bt},\qquad
 \|n_j^i(t)\|_{L^1(\mathbb{R}^d)} \le e^{Bt}\|n_j^{i,\rm in}\|_{L^1(\mathbb{R}^d)} \le C_{\rm in}|K|e^{Bt}.
\]

The initial species bound follows from $n_j^{i,\rm in}\le C_{\rm in}\mathbf1_K$. For the pressure mass, the density bound yields:
\[
\begin{aligned}
 p_j \le\frac{k_j}{k_j-1}b_j^{k_j-2}\rho_j =\left(\frac{k_j}{k_j-1}\right)^{1/(k_j-1)}
 P^{(k_j-2)/(k_j-1)}\rho_j
 \le C_p\rho_j.
\end{aligned}
\]

Hence $\|p_j(t)\|_1\le C_pM_0e^{Bt}$. Young's inequality, with kernel norm exactly one for every
$\nu_j$, gives:
\[
 \|W_j(t)\|_1\le\|p_j(t)\|_1\le C_pM_0e^{Bt},
 \qquad
 \|W_j(t)\|_\infty\le\|p_j(t)\|_\infty\le P.
\]

This proves~\eqref{mass-bounds}, and interpolation gives the claimed $L^q(Q_T)$ bounds for $\rho_j,p_j,W_j$.

For the individual $L^2$ estimate, boundedness and uniform trait Lipschitz continuity of $n_0$ give:
\[
\begin{aligned}
 (n_j^{i,\rm in})^2
 &\le\frac1{N_j}\sum_{\ell=1}^{N_j}(n_j^{\ell,\rm in})^2
   +\sum_{\ell=1}^{N_j-1}
    |(n_j^{\ell+1,\rm in})^2-(n_j^{\ell,\rm in})^2|\\
 &\le C_{\rm in}\rho_j^{\rm in}
   +\frac{L_{\rm in}}{N_j}
      \sum_{\ell=1}^{N_j-1}
      (n_j^{\ell+1,\rm in}+n_j^{\ell,\rm in})
 \le C_*\rho_j^{\rm in}.
\end{aligned}
\]

Consequently,
\[
 \int_{\mathbb R^d}
 \frac{(n_j^{i,\rm in})^2}{\rho_j^{\rm in}}\,dx
 \le C_*|K|,
\]
where the quotient is defined as zero on
$\{\rho_j^{\rm in}=0\}$.

To propagate this estimate, fix $j,i$ and $\varepsilon>0$,
and set
\[
 q_\varepsilon=\frac{(n_j^i)^2}{\rho_j+\varepsilon}.
\]
We justify its equation by applying the mollification
argument from Lemma~\ref{lem:weak-solution-chain-rules}
simultaneously to the species equation and the
total-density equation. For the function
$(s,r)\mapsto s^2/(r+\varepsilon)$, the relevant derivatives are
\[
 \partial_s\frac{s^2}{r+\varepsilon}
 =\frac{2s}{r+\varepsilon},
 \qquad
 \partial_r\frac{s^2}{r+\varepsilon}
 =-\frac{s^2}{(r+\varepsilon)^2}.
\]
At fixed $j,\varepsilon$, these derivatives are bounded
on the ranges of the nonnegative mollified densities.
Each mollified equation has a commutator error tending
to zero in $L^1$, as proved in the appendix. Multiplication
by the corresponding bounded derivative therefore leaves
an error tending to zero in $L^1$. Applying the ordinary
two-variable chain rule to the mollified densities and
then removing the mollification is consequently justified.

Both equations have the same velocity $\nabla W_j$.
The coefficient of $\Delta W_j$ in this calculation is
\[
 s\frac{2s}{r+\varepsilon}
 -r\frac{s^2}{(r+\varepsilon)^2}
 -\frac{s^2}{r+\varepsilon}
 =\frac{\varepsilon s^2}{(r+\varepsilon)^2}.
\]
The species and total-density reaction terms give,
respectively,
\[
 \frac{2n_j^i}{\rho_j+\varepsilon}
       n_j^iG(p_j,a_{j,i});
 \qquad
 -\frac{(n_j^i)^2}{(\rho_j+\varepsilon)^2}\rho_jA_j.
\]
Thus
\[
\begin{aligned}
 \partial_t q_\varepsilon
 &=\operatorname{div}(q_\varepsilon\nabla W_j)
   +\left(2G(p_j,a_{j,i})
        -\frac{\rho_j}{\rho_j+\varepsilon}A_j\right)
      q_\varepsilon +\varepsilon
       \frac{(n_j^i)^2}{(\rho_j+\varepsilon)^2}\Delta W_j.
\end{aligned}
\]
The same calculation includes the initial term
\[
 q_\varepsilon(0)
 =\frac{(n_j^{i,\rm in})^2}
        {\rho_j^{\rm in}+\varepsilon}.
\]
Indeed, at fixed $\varepsilon$ the two-variable function
is Lipschitz on the bounded ranges involved, so the
prescribed $L^1$ initial traces pass through this function.

We now let $\varepsilon\downarrow0$ with $j$ fixed.
Nonnegativity and the definition of $\rho_j$ give
$n_j^i\le N_j\rho_j$. Hence
\[
 0\le
 \varepsilon\frac{(n_j^i)^2}{(\rho_j+\varepsilon)^2}
 \le
 N_j^2\frac{\varepsilon\rho_j^2}
             {(\rho_j+\varepsilon)^2}
 \le\frac{N_j^2}{4}\rho_j.
\]
The expression on the left tends pointwise to zero,
including on the set where $\rho_j=0$, where $n_j^i=0$.
Since $\Delta W_j$ is bounded at fixed $j$ and
$\rho_j\in L^1(Q_T)$, dominated convergence gives
\[
 \int_{Q_T}
 \varepsilon\frac{(n_j^i)^2}{(\rho_j+\varepsilon)^2}
       |\Delta W_j|\,dx\,dt
 \longrightarrow0.
\]
Furthermore,
\[
 q_\varepsilon\uparrow q_{j,i}
 :=\frac{(n_j^i)^2}{\rho_j},
 \qquad
 0\le q_\varepsilon\le q_{j,i}\le N_jn_j^i,
\]
where $q_{j,i}=0$ on $\{\rho_j=0\}$.
This gives convergence in $L^1(Q_T)$.
The reaction terms also converge in $L^1(Q_T)$:
their absolute values are bounded by $3Bq_{j,i}$,
and their pointwise limit is
$(2G(p_j,a_{j,i})-A_j)q_{j,i}$.
The initial quotients converge in $L^1(\mathbb R^d)$
by the corresponding bound
\[
 \frac{(n_j^{i,\rm in})^2}{\rho_j^{\rm in}}
 \le N_jn_j^{i,\rm in}.
\]
We therefore obtain, including the initial term,
\[
 \partial_t q_{j,i}
 =\operatorname{div}(q_{j,i}\nabla W_j)
  +(2G(p_j,a_{j,i})-A_j)q_{j,i}.
\]

The term involving $\Delta W_j$ has thus vanished before
we take any uniform estimate. Since $\nabla W_j$ is
bounded at fixed $j$ and $q_{j,i}\le N_jn_j^i$ is
integrable, spatial integration removes the transport
divergence. Using $2G(p_j,a_{j,i})-A_j\le3B$ gives
\[
 \int_{\mathbb R^d}q_{j,i}(t)\,dx
 \le C_*|K|
   +3B\int_0^t\int_{\mathbb R^d}q_{j,i}(s)\,dx\,ds.
\]
Gronwall's inequality and $\rho_j\le C_\rho$ imply
\[
 \int_{\mathbb R^d}q_{j,i}(t)\,dx
 \le C_*|K|e^{3Bt},
 \qquad
 \|n_j^i(t)\|_2^2
 =\int_{\mathbb R^d}\rho_jq_{j,i}(t)\,dx
 \le C_\rho C_*|K|e^{3Bt}
\]
for almost every $t\in(0,T)$.
This proves~\eqref{eq:species-quadratic-bound}.
Finally, the reconstruction identity
\[
 \|u_j(t)\|_{L^2(\mathbb R^d\times I)}^2
 =\frac1{N_j}\sum_{i=1}^{N_j}\|n_j^i(t)\|_2^2
 \le C_\rho C_*|K|e^{3Bt}
\]
proves~\eqref{eq:trait-joint-quadratic-bound}.
Interpolation between the individual $L^1$ and $L^2$
bounds completes the proof.
\end{proof}

\subsection{Dissipation and tail estimates}

We now establish some important dissipation estimate results.

\begin{lemma}[Full pressure dissipation]\label{lem:pressure-dissipation}
For almost every $s\in(0,T)$,
\begin{equation}\label{pressure-energy}
 \frac{1}{k_j-2}\int_{\R^d}p_j(s)\dd x
 +\int_0^s\!\!\int_{\R^d}
       \bigl(|\nabla W_j|^2+\nu_j|\Delta W_j|^2\bigr)
 =\frac{1}{k_j-2}\int_{\R^d}p_j^{\rm in}\dd x
 +\frac{k_j-1}{k_j-2}\int_0^s\!\!\int_{\R^d}p_jA_j.
\end{equation}
In particular,
\begin{equation}\label{basic-dissipation}
 \sup_j\iint_{Q_T}\bigl(|\nabla W_j|^2+\nu_j|\Delta W_j|^2\bigr)<\infty,
 \qquad \|p_j-W_j\|_{L^2(Q_T)}\le C\sqrt{\nu_j}.
\end{equation}
After dropping the nonnegative terminal term, \eqref{pressure-energy} holds as an inequality at $s=T$ as well.
\end{lemma}

\begin{proof}
For $e_j(r)=P_j(r)/(k_j-2)$, the power pressure law and $R_j=\rho_jA_j$ give:
\[
 \rho_je_j'(\rho_j)-e_j(\rho_j)=p_j,
 \qquad
 e_j'(\rho_j)R_j
 =\frac{k_j-1}{k_j-2}p_jA_j.
\]

Applying Lemma~\ref{lem:weak-solution-chain-rules} therefore yields:
\[
 \partial_t\frac{p_j}{k_j-2}
 -\operatorname{div}\left(
       \frac{p_j}{k_j-2}\nabla W_j\right)
 =p_j\Delta W_j+\frac{k_j-1}{k_j-2}p_jA_j.
\]

Integrating over $\mathbb R^d$, the divergence term vanishes, and we obtain:
\[
 \frac{1}{k_j-2}\frac{d}{dt}
       \int_{\mathbb R^d}p_j\,dx
 =\int_{\mathbb R^d}p_j\Delta W_j\,dx
  +\frac{k_j-1}{k_j-2}
       \int_{\mathbb R^d}p_jA_j\,dx.
\]

At fixed $j$, the bound $p_j\in L^2(Q_T)$ and the convolutional formula give $W_j\in L^2(0,T;H^2(\mathbb R^d))$. This suffices to justify integration by parts. Using $p_j=W_j-\nu_j\Delta W_j$, we calculate:
\[
\begin{aligned}
 \int_{\mathbb R^d}p_j\Delta W_j\,dx
 &=\int_{\mathbb R^d}
       (W_j-\nu_j\Delta W_j)\Delta W_j\,dx\\
 &=-\int_{\mathbb R^d}|\nabla W_j|^2\,dx
   -\nu_j\int_{\mathbb R^d}|\Delta W_j|^2\,dx.
\end{aligned}
\]

Substituting this into the preceding balance and integrating in time from $0$ to $s$, with the prescribed initial trace, gives, for almost every $s\in(0,T)$,
\[
\begin{aligned}
 &\frac{1}{k_j-2}\int_{\mathbb R^d}p_j(s)\,dx
 +\int_0^s\int_{\mathbb R^d}
       \bigl(|\nabla W_j|^2+\nu_j|\Delta W_j|^2\bigr)\,dx\,dt\\
 &\qquad=
 \frac{1}{k_j-2}\int_{\mathbb R^d}p_j^{\rm in}\,dx
 +\frac{k_j-1}{k_j-2}
       \int_0^s\int_{\mathbb R^d}p_jA_j\,dx\,dt.
\end{aligned}
\]

This proves~\eqref{pressure-energy}. To obtain the terminal inequality, choose times $s_m\uparrow T$
at which this identity holds and discard the nonnegative term involving $p_j(s_m)$. The dissipation integrals increase to the integral over $Q_T$, while the reaction integrals converge because $|p_jA_j|\le Bp_j\in L^1(Q_T)$. Hence:
\[
 \iint_{Q_T}
       \bigl(|\nabla W_j|^2+\nu_j|\Delta W_j|^2\bigr)
 \le
 \frac{1}{k_j-2}\int_{\mathbb R^d}p_j^{\rm in}\,dx
 +\frac{k_j-1}{k_j-2}\iint_{Q_T}p_jA_j.
\]

The preceding a-priori estimates give:
\[
 \|p_j^{\rm in}\|_1\le C_pM_0,\qquad
 \sup_t\|p_j(t)\|_1\le C_pM_0e^{BT}.
\]

Since $k_j\ge3$ and $|A_j|\le B$, it follows that:
\[
\begin{aligned}
 \iint_{Q_T}
       \bigl(|\nabla W_j|^2+\nu_j|\Delta W_j|^2\bigr)
 &\le
 \frac{C_pM_0}{k_j-2}
 +\frac{k_j-1}{k_j-2}BTC_pM_0e^{BT}\\
 &\le C_pM_0(1+2BT e^{BT})=:C_D.
\end{aligned}
\]

The constant $C_D$ is independent of $j$. Finally, $p_j-W_j=-\nu_j\Delta W_j$ gives:
\[
 \|p_j-W_j\|_{L^2(Q_T)}^2
 =\nu_j^2\iint_{Q_T}|\Delta W_j|^2
 \le \nu_j C_D.
\]

Thus $\|p_j-W_j\|_{L^2(Q_T)}\le\sqrt{C_D}\sqrt{\nu_j}$, which completes~\eqref{basic-dissipation}.
\end{proof}

\begin{lemma}[Time compactness and tail estimates]\label{lem:time-tails}
We have
\begin{equation}\label{time-bounds}
 \|\partial_t\rho_j\|_{L^2(0,T;H^{-1}(\R^d))}\le C.
\end{equation}
Moreover,
\begin{equation}\label{moment-bound}
 \sup_{j,t}\int_{\R^d}|x|\rho_j(x,t)\dd x\le C,
 \qquad \sup_{j,t}\int_{|x|>L}\rho_j(x,t)\dd x\le\frac CL.
\end{equation}
\end{lemma}

\begin{proof}
By \eqref{pressure-bounds}, \eqref{mass-bounds}, and
\eqref{basic-dissipation},
\[
 \|\nabla W_j\|_{L^2(Q_T)}^2\le C_D,
 \qquad
 \|\rho_j\|_{L^2(Q_T)}^2
 \le C_\rho\int_0^T\|\rho_j(t)\|_1\,dt
 \le TC_\rho M_0e^{BT}.
\]
We estimate $\partial_t\rho_j$ directly from \eqref{total-eq}.
For $\varphi\in C_c^\infty(\mathbb R^d)$, that equation gives
\[
 \langle\partial_t\rho_j,\varphi\rangle
 =-\int_{\mathbb R^d}
       \rho_j\nabla W_j\cdot\nabla\varphi\,dx
   +\int_{\mathbb R^d}R_j\varphi\,dx.
\]
Using $\rho_j\le C_\rho$, $|R_j|\le B\rho_j$,
and Cauchy--Schwarz, we obtain
\[
\begin{aligned}
 |\langle\partial_t\rho_j,\varphi\rangle|
 &\le C_\rho\|\nabla W_j(t)\|_2\|\nabla\varphi\|_2
       +B\|\rho_j(t)\|_2\|\varphi\|_2\\
 &\le
 \left(C_\rho\|\nabla W_j(t)\|_2+B\|\rho_j(t)\|_2\right)
 \|\varphi\|_{H^1}.
\end{aligned}
\]
Taking the supremum over tests with $\|\varphi\|_{H^1}\le1$
and then the $L^2$ norm in time yields
\[
\begin{aligned}
 \|\partial_t\rho_j\|_{L^2(0,T;H^{-1})}
 &\le C_\rho\|\nabla W_j\|_{L^2(Q_T)}
       +B\|\rho_j\|_{L^2(Q_T)}\\
 &\le C_\rho C_D^{1/2}
       +B(TC_\rho M_0e^{BT})^{1/2}
 =:C.
\end{aligned}
\]
This proves~\eqref{time-bounds}, with $C$ independent of $j$.

To estimate the first spatial moment, let
\[
 \psi_R(x)=\min\{|x|,R\}.
\]
These bounded weights satisfy
\[
 0\le\psi_R\le|x|,\qquad
 |\nabla\psi_R|\le1\quad\text{a.e.},\qquad
 \psi_R\uparrow|x|\quad\text{as }R\to\infty.
\]
Integrating the total-density equation by parts gives
\[
\begin{aligned}
 \frac{d}{dt}\int_{\mathbb R^d}\psi_R\rho_j\,dx
 &=-\int_{\mathbb R^d}
       \rho_j\nabla W_j\cdot\nabla\psi_R\,dx
   +\int_{\mathbb R^d}\psi_RR_j\,dx\\
 &\le\|\rho_j(t)\|_2\|\nabla W_j(t)\|_2
   +B\int_{\mathbb R^d}\psi_R\rho_j\,dx.
\end{aligned}
\]
Here the integrations by parts are justified by smooth
approximation and passage to the limit in the resulting
integrals.
Integrating in time and using
$\int\psi_R\rho_j^{\rm in}\,dx
 \le\int|x|\rho_j^{\rm in}\,dx\le M_0$, we obtain
\[
\begin{aligned}
 \int_{\mathbb R^d}\psi_R\rho_j(t)\,dx
 &\le M_0
   +\int_0^t\|\rho_j(s)\|_2\|\nabla W_j(s)\|_2\,ds\\
 &\quad+B\int_0^t\int_{\mathbb R^d}\psi_R\rho_j\,dx\,ds.
\end{aligned}
\]
The preceding density and dissipation bounds give
\[
 \sup_s\|\rho_j(s)\|_2^2\le C_\rho M_0e^{BT},
 \qquad
 \|\nabla W_j\|_{L^2(Q_T)}^2\le C_D.
\]
Consequently, Cauchy--Schwarz in time yields
\[
\begin{aligned}
 \int_0^t\|\rho_j(s)\|_2\|\nabla W_j(s)\|_2\,ds
 &\le\sqrt{T}\,
       \sup_s\|\rho_j(s)\|_2
       \|\nabla W_j\|_{L^2(Q_T)}\\
 &\le(TC_\rho M_0e^{BT}C_D)^{1/2}
 =:C_{\rm tr}.
\end{aligned}
\]
Gronwall's inequality therefore gives
\[
 \int_{\mathbb R^d}\psi_R(x)\rho_j(x,t)\,dx
 \le(M_0+C_{\rm tr})e^{Bt}
 \le(M_0+C_{\rm tr})e^{BT}=:C_M.
\]
The bound is independent of both $j$ and $R$.
Letting $R\to\infty$, the weights increase to $|x|$
from below, so monotone convergence theorem gives
\[
 \int_{\mathbb R^d}|x|\rho_j(x,t)\,dx
 =\lim_{R\to\infty}
   \int_{\mathbb R^d}\psi_R(x)\rho_j(x,t)\,dx
 \le C_M.
\]

Finally, on $\{|x|>L\}$ we have $1\le|x|/L$, so
\[
 \int_{|x|>L}\rho_j(x,t)\,dx
 \le\frac1L\int_{|x|>L}|x|\rho_j(x,t)\,dx
 \le\frac{C_M}{L}.
\]
This proves both assertions in~\eqref{moment-bound}.
\end{proof}

\subsection{Weak limits}
The preceding bounds permit a common subsequence such that
\begin{equation}\label{weak-limits}
 u_j\rightharpoonup n\text{ in }L^2(Q_T\times I),\quad \rho_j\rightharpoonup\rho,\quad
 R_j\rightharpoonup R,\quad p_j\rightharpoonup p,\quad
 W_j\rightharpoonup p\text{ in }L^2(0,T;H^1(\R^d)).
\end{equation}
The weak convergences not otherwise specified are weak convergences on $L^2(Q_T)$; quantities with uniform $L^{\infty}(Q_T)$ bounds can also have a common subsequence that converges in weak-star in $L^\infty(Q_T)$. The limit $R$ is only a notation for the weak reaction limit at this stage.

By passing the inequalities in proposition \ref{prop:bounds} through the subsequence, we have:

\begin{align*}
    \rho=\int_I n\dd a \text{, } 0\le p\le P \text{, and } |R|\le B\rho
\end{align*}

The mass and moment bounds pass to the limits by testing those inequalities against non-negative test functions. In particular, for these weak limits, we have: $\rho,p,R\in L^2(Q_T)$ and $p\in L^\infty(0,T;L^1(\R^d))$.

\section{Identification before strong convergence}\label{sec:weak-identification}
\subsection{Weak products and the congestion constraint}

The arguments below precede every strong convergence assertion. We use the
bounds established above and the fixed subsequence. Recall
\[
 u_j\rightharpoonup n\quad\hbox{in }L^2(Q_T\times(0,1)),\qquad
 \rho_j\stackrel{*}{\rightharpoonup}\rho,\qquad
 R_j\rightharpoonup R,\qquad
 W_j\rightharpoonup p\quad\hbox{in }L^2(0,T;H^1(\mathbb R^d)).
\]
Here $R$ is, for the moment, an unidentified weak limit. Since
$\|p_j-W_j\|_{L^2(Q_T)}\to0$, the weak limit of $p_j$ is also $p$.
In particular, $p\geq0$. Throughout this subsection test functions have compact
support in $\mathbb R^d\times(0,T)$.

\begin{lemma}[Weak limits of product]\label{lem:weak-products-revised}
For every fixed $b\in L^\infty(0,1)$, set
$m_j^b=\int_0^1u_j(\cdot,a)b(a)\,da$ and
$m^b=\int_0^1n(\cdot,a)b(a)\,da$. Then
\[
 m_j^bp_j\rightharpoonup m^bp\quad\hbox{in distributions}.
\]
Consequently,
\[
 \rho_jp_j\rightharpoonup\rho p,\quad\hbox{locally in }L^2(Q_T); \quad
 u_jp_j\rightharpoonup np
 \quad\hbox{locally in }L^2(Q_T\times(0,1)).
\]
The last statement means convergence against every $L^2$ test function with
compact support in the space--time variables; it does not assert strong
convergence of either factor.
\end{lemma}

\begin{proof}
We work along the subsequence fixed in \eqref{weak-limits}. In particular,
\[
 u_j\rightharpoonup n
 \quad\text{in }L^2(Q_T\times(0,1)),
 \qquad
 W_j\rightharpoonup p
 \quad\text{in }L^2(0,T;H^1(\mathbb R^d)),
\]
and \eqref{basic-dissipation} gives $\|p_j-W_j\|_{L^2(Q_T)}\to0$.

Fix $b\in L^\infty(0,1)$ and a bounded smooth spatial domain $\Omega\subset\mathbb R^d$. Multiplying each species $n_j^i$ equation by $b(a)$, integrating over its phenotype cell $I_{j,i}$, and summing over $i$ gives:
\[
 \partial_tm_j^b
 =\operatorname{div}(m_j^b\nabla W_j)+S_j^b,
 \qquad
 S_j^b=\int_0^1b(a)u_j(a)G_j(p_j,a)\,da.
\]
Nonnegativity of $u_j$ and the growth bound imply
\[
 |m_j^b|\le\|b\|_\infty\rho_j,
 \qquad
 |S_j^b|\le B\|b\|_\infty\rho_j.
\]
Thus $m_j^b$ is uniformly bounded in
$L^2([0,T];L^2(\Omega))$.
For $\psi\in H_0^1(\Omega)$, integration by parts and
Cauchy--Schwarz give
\[
\begin{aligned}
 |\langle\partial_tm_j^b,\psi\rangle|
 &=
 \left|-\int_\Omega m_j^b\nabla W_j\cdot\nabla\psi\,dx
       +\int_\Omega S_j^b\psi\,dx\right|\\
 &\le\|b\|_\infty
 \left(C_\rho\|\nabla W_j(t)\|_{L^2(\Omega)}
       +B\|\rho_j(t)\|_{L^2(\Omega)}\right)
 \|\psi\|_{H^1(\Omega)}.
\end{aligned}
\]
Taking the dual norm and then integrating in time yields
\[
 \|\partial_tm_j^b\|_{L^2([0,T];H^{-1}(\Omega))}
 \le\|b\|_\infty
 \left(C_\rho\|\nabla W_j\|_{L^2(Q_T)}
       +B\|\rho_j\|_{L^2(Q_T)}\right)
 \le C\|b\|_\infty.
\]

Now for each fixed $b\in L^\infty(I)$, the preceding estimates give
\[
 \sup_j\left(
 \|m_j^b\|_{L^2(0,T;L^2(\Omega))}
 +\|\partial_t m_j^b\|_{L^2(0,T;H^{-1}(\Omega))}
 \right)<\infty.
\]
Since
\[
 L^2(\Omega)\Subset H^{-1}(\Omega),
\]
the Aubin--Lions compactness theorem in the form of
\cite[Corollary~4, p.~85]{Simon1987}, with
$X=L^2(\Omega)$ and $B=Y=H^{-1}(\Omega)$, applies to
$(m_j^b)$. More precisely, every subsequence admits a further
subsequence $(m_{j_\ell}^b)$ and a function
$\widetilde m\in L^2(0,T;H^{-1}(\Omega))$ such that
\[
 \lim_{\ell\to\infty}
 \int_0^T
 \|m_{j_\ell}^b(t)-\widetilde m(t)\|_{H^{-1}(\Omega)}^2\,dt
 =0.
\]

We identify $\widetilde m$ using the already established weak
convergence of $u_j$ in $L^2(Q_T\times I)$.
For every $\zeta\in C_c^\infty(\Omega\times(0,T))$,
\[
\begin{aligned}
 \int_0^T\langle\widetilde m(t),\zeta(t)\rangle\,dt
 &=\lim_{\ell\to\infty}
   \int_0^T\int_\Omega m_{j_\ell}^b\zeta\,dx\,dt\\
 &=\lim_{\ell\to\infty}
   \int_0^T\int_\Omega\int_I
       u_{j_\ell}(x,t,a)b(a)\zeta(x,t)\,da\,dx\,dt\\
 &=\int_0^T\int_\Omega\int_I
       n(x,t,a)b(a)\zeta(x,t)\,da\,dx\,dt\\
 &=\int_0^T\int_\Omega m^b\zeta\,dx\,dt.
\end{aligned}
\]
The first equality follows from the displayed norm limit
and $H^{-1}$--$H_0^1$ duality; the third follows because
$b(a)\zeta(x,t)$ is a fixed $L^2$ test function.
Thus $\widetilde m=m^b$.

Thus, for every fixed $b$, every subsequence therefore admits a further subsequence for which the same norm limit holds with limit $m^b$. Consequently, along the already fixed sequence,
\[
 \lim_{j\to\infty}
 \int_0^T
 \|m_j^b(t)-m^b(t)\|_{H^{-1}(\Omega)}^2\,dt=0.
\]

Now let $\varphi\in C_c^\infty(\Omega\times(0,T))$.
We decompose the product difference as
\[
\begin{aligned}
 \iint\varphi(m_j^bp_j-m^bp)
 &=\iint\varphi m_j^b(p_j-W_j) +\iint\varphi(m_j^b-m^b)W_j
       +\iint\varphi m^b(W_j-p).
\end{aligned}
\]
All integrals here are over $\Omega\times(0,T)$.
The first term tends to zero because
\[
 \left|\iint\varphi m_j^b(p_j-W_j)\right|
 \le\|\varphi\|_\infty
      \|m_j^b\|_{L^2(\Omega\times(0,T))}
      \|p_j-W_j\|_{L^2(Q_T)}
 \longrightarrow0.
\]
For the second term, compact spatial support of $\varphi$
ensures that $\varphi W_j\in L^2(0,T;H_0^1(\Omega))$.
The product rule gives
\[
 \|\varphi W_j\|_{L^2(0,T;H_0^1(\Omega))}
 \le C_\varphi\|W_j\|_{L^2(0,T;H^1(\Omega))}
 \le C_\varphi.
\]
Therefore, the $H^{-1}$--$H_0^1$ duality inequality gives
\[
\begin{aligned}
 \left|\iint\varphi(m_j^b-m^b)W_j\right| =
 \left|\int_0^T
 \langle m_j^b-m^b,\varphi W_j\rangle\,dt\right| \le
 \|m_j^b-m^b\|_{L^2(0,T;H^{-1}(\Omega))}
 \|\varphi W_j\|_{L^2(0,T;H_0^1(\Omega))}.
\end{aligned}
\]

And we see that $\|m_j^b-m^b\|_{L^2(0,T;H^{-1}(\Omega))} \|\varphi W_j\|_{L^2(0,T;H_0^1(\Omega))} \longrightarrow 0$ by all the estimates above.

For the third term, the bound $|m^b|\le\|b\|_\infty\rho\le C_{\rho}\|b\|_\infty$ shows that $\varphi m^b$ is a fixed bounded function with compact support, hence belongs to $L^2(Q_T)$. Weak convergence of $W_j$ to $p$ in $L^2(Q_T)$ therefore gives:

\[
 \iint_{Q_T}\varphi m^b(W_j-p)\,dx\,dt\longrightarrow0.
\]

Combining the three limits proves
$m_j^bp_j\rightharpoonup m^bp$ in distributions.

Taking $b=1$ gives $m_j^1=\rho_j$ and $m^1=\rho$,
so $\rho_jp_j\rightharpoonup\rho p$ in distributions.
Moreover,
\[
 \|\rho_jp_j\|_{L^2(Q_T)}
 \le P\|\rho_j\|_{L^2(Q_T)}\le C,
 \qquad \rho p\in L^2(Q_T).
\]

The uniform bound
\[
 \|\rho_jp_j\|_{L^2(Q_T)}
 \le P\|\rho_j\|_{L^2(Q_T)}\le C
\]
and density of smooth tests give
$\rho_jp_j\rightharpoonup\rho p$ in $L^2_{\rm loc}(Q_T)$.

For the phenotype product, the first assertion gives
\[
 \int_{Q_T\times I}(u_jp_j-np)\varphi(x,t)b(a)
 =\iint_{Q_T}\varphi(m_j^bp_j-m^bp)\longrightarrow0
\]
for every $\varphi\in C_c^\infty(Q_T)$ and
$b\in L^\infty(I)$. Moreover,
\[
 \sup_j\|u_jp_j\|_{L^2(Q_T\times I)}
 \le P\sup_j\|u_j\|_{L^2(Q_T\times I)}<\infty,
 \qquad
 \|np\|_{L^2(Q_T\times I)}\le P\|n\|_2.
\]
For every bounded cylinder $D\Subset Q_T$, finite sums of
$\varphi(x,t)b(a)$, with $\varphi\in C_c^\infty(D)$ and
$b\in L^\infty(I)$, are dense in $L^2(D\times I)$.
These bounds therefore extend the convergence to every
$L^2(D\times I)$ test, proving the asserted joint
weak convergence.

\end{proof}

\begin{lemma}[Complementarity relation]\label{lem:congestion-revised}
The weak limits satisfy
\[
 \rho=\int_0^1n(\cdot,a)\,da,\qquad
 0\leq\rho\leq1,\qquad p\geq0,\qquad p(1-\rho)=0
 \quad\hbox{almost everywhere}.
\]
\end{lemma}

\begin{proof}
For every $\zeta\in C_c^\infty(Q_T)$, the reconstruction
identity and the established weak convergences give
\[
\begin{aligned}
 \iint_{Q_T}\rho\zeta
 &=\lim_{j\to\infty}\iint_{Q_T}\rho_j\zeta\\
 &=\lim_{j\to\infty}
   \int_{Q_T\times I}u_j(x,t,a)\zeta(x,t)\,dx\,dt\,da\\
 &=\int_{Q_T\times I}n(x,t,a)\zeta(x,t)\,dx\,dt\,da.
\end{aligned}
\]
Hence $\rho=\int_I n\,da$ almost everywhere.

Let $P>0$ be the uniform pressure bound. Inverting the
pressure law yields
\[
 0\le\rho_j
 \le b_j:=
 \left(\frac{k_j-1}{k_j}P\right)^{1/(k_j-1)}.
\]
Since
\[
 \log b_j
 =\frac{\log((k_j-1)/k_j)+\log P}{k_j-1}
 \longrightarrow0,
\]
we have $b_j\to1$. For every nonnegative
$\zeta\in C_c^\infty(Q_T)$, it follows that
\[
 0\le\iint\rho\zeta
 =\lim_{j\to\infty}\iint\rho_j\zeta
 \le\lim_{j\to\infty}b_j\iint\zeta
 =\iint\zeta.
\]
Thus $0\le\rho\le1$ almost everywhere.

To prove complementarity, we estimate
$p_j|\rho_j-1|$ directly. On $\{p_j>0\}$,
\[
 \rho_j
 =\exp\left(
 \frac{\log((k_j-1)/k_j)+\log p_j}{k_j-1}
 \right).
\]
The elementary inequality
$|e^s-1|\le\max\{1,e^s\}|s|$, together with
$\rho_j\le C_\rho$ and $k_j\ge3$, gives
\[
 p_j|\rho_j-1|
 \le\frac{C_\rho}{k_j-1}
 \left(p_j|\log p_j|+p_j\log\frac32\right).
\]
For $0<p_j\le1$, we have $p_j|\log p_j|\le1/e$;
for $1\le p_j\le P$, we have
$p_j|\log p_j|\le P\log(\max\{1,P\})$.
Consequently,
\[
 p_j|\rho_j-1|
 \le\frac{C_P}{k_j-1},
 \qquad
 C_P=C_\rho\left(
 \frac1e+P\log(\max\{1,P\})+P\log\frac32
 \right).
\]
This bound also holds on $\{p_j=0\}$, where the left-hand
side is zero. Therefore, for every measurable set
$D\subset Q_T$ of finite measure,
\[
 \int_D|\rho_jp_j-p_j|
 \le\frac{C_P|D|}{k_j-1}\longrightarrow0.
\]

Finally, Lemma~\ref{lem:weak-products-revised} identifies
the limit of $\rho_jp_j$ as $\rho p$, while the previously
established weak convergence identifies the limit of
$p_j$ as $p$. Thus, for every
$\zeta\in C_c^\infty(Q_T)$,
\[
 \iint\zeta(\rho p-p)
 =\lim_{j\to\infty}\iint\zeta(\rho_jp_j-p_j)=0,
\]
where the last limit follows from
\[
 \left|\iint\zeta(\rho_jp_j-p_j)\right|
 \le\frac{C_P}{k_j-1}\|\zeta\|_{L^1(Q_T)}.
\]
Hence $p(1-\rho)=0$ almost everywhere, completing the proof.
\end{proof}

\subsection{The total-density flux before strong convergence}

The following truncation estimate provides the missing noncircular
identification of the total-density equation. In particular, it uses neither
strong pressure convergence nor an identification of the reaction limit.

\begin{lemma}[Dissipation at small potential]\label{lem:small-potential-revised}
For every nonnegative $\varphi\in C_c^\infty(Q_T)$ there is a constant
$C_\varphi$, independent of $\delta\in(0,P)$, such that
\[
 \limsup_{j\to\infty}
 \iint\varphi\mathbf1_{\{W_j<\delta\}}|\nabla W_j|^2
 \leq C_\varphi\delta.
\]
Here and below $Q_T=\mathbb R^d\times(0,T)$ when used as the domain of a
test function.
\end{lemma}

\begin{proof}
Fix $\delta\in(0,P)$ and restrict to indices with $k_j\ge3$.
Define
\[
 r_{j,\delta}
 =\left(\frac{k_j-1}{k_j}\delta\right)^{1/(k_j-1)}
\]
and
\[
 e_{j,\delta}(r)=
 \begin{cases}
 \displaystyle\frac{P_j(r)}{k_j-2},
     &0\le r\le r_{j,\delta},\\[6pt]
 \displaystyle\delta\left(
   \frac{k_j-1}{k_j-2}\frac{r}{r_{j,\delta}}-1
 \right),
     &r\ge r_{j,\delta}.
 \end{cases}
\]
The two formulas agree at $r=r_{j,\delta}$, where their
derivatives both equal
\[
 \frac{(k_j-1)\delta}{(k_j-2)r_{j,\delta}}.
\]
Also $e_{j,\delta}(0)=e_{j,\delta}'(0)=0$.
Thus extension by zero to negative arguments gives a
$C^1(\mathbb R)$ function. Direct differentiation gives
\[
 re_{j,\delta}'(r)-e_{j,\delta}(r)
 =\min\{P_j(r),\delta\}
 \qquad(r\ge0).
\]
The function $e_{j,\delta}$ is nonnegative and nondecreasing.
Using $\rho_j\le b_j$ and $P_j(b_j)=P>\delta$, we obtain
\[
\begin{aligned}
 0\le e_{j,\delta}(\rho_j) \le e_{j,\delta}(b_j) =\delta\left[
   \frac{k_j-1}{k_j-2}
   \left(\frac P\delta\right)^{1/(k_j-1)}-1
 \right]
 =:\varepsilon_{j,\delta}.
\end{aligned}
\]
For every fixed $\delta$, we have
$\varepsilon_{j,\delta}\to0$ as $j\to\infty$.

Applying Lemma~\ref{lem:weak-solution-chain-rules}
to $e_{j,\delta}$ and using $R_j=\rho_jA_j$ gives:
\[
\begin{aligned}
 \partial_te_{j,\delta}(\rho_j)
 ={}&\operatorname{div}
       \bigl(e_{j,\delta}(\rho_j)\nabla W_j\bigr)
   +\min\{p_j,\delta\}\Delta W_j +[e_{j,\delta}(\rho_j)+\min\{p_j,\delta\}]A_j.
\end{aligned}
\]
To rewrite the Laplacian term $\min\{p_j,\delta\}\Delta W_j$, define
\[
 z_\delta(q)=
 \begin{cases}
 q^2/2,&0\le q\le\delta,\\
 \delta q-\delta^2/2,&q\ge\delta.
 \end{cases}
\]
Then $z_\delta\in C^1([0,\infty))$ and
$z_\delta'(q)=\min\{q,\delta\}$.
The Sobolev chain rule gives
\[
 \nabla z_\delta(W_j)=\min\{W_j,\delta\}\nabla W_j,
 \qquad
 \nabla\min\{W_j,\delta\}
 =\mathbf1_{\{W_j<\delta\}}\nabla W_j
 \quad\text{a.e.}
\]
The second identity is unaffected by the corner at $\delta$:
$\nabla W_j=0$ almost everywhere on $\{W_j=\delta\}$.
The fixed-$j$ regularity
$W_j\in L^2(0,T;H^2(\mathbb R^d))$
therefore justifies the distributional product rule
\[
 \Delta z_\delta(W_j)
 =\min\{W_j,\delta\}\Delta W_j
  +\mathbf1_{\{W_j<\delta\}}|\nabla W_j|^2.
\]

Since the mapping $q\mapsto\min\{q,\delta\}$ is nondecreasing,
the Brinkman equation implies
\[
\begin{aligned}
 [\min\{p_j,\delta\}-\min\{W_j,\delta\}]\Delta W_j = -\frac{
 [\min\{p_j,\delta\}-\min\{W_j,\delta\}](p_j-W_j)}
 {\nu_j}
 \le0.
\end{aligned}
\]
Consequently,
\[
 \min\{p_j,\delta\}\Delta W_j
 \le\Delta z_\delta(W_j)
    -\mathbf1_{\{W_j<\delta\}}|\nabla W_j|^2
\]
in distributions. Substitute this inequality into the
entropy equation, multiply by
$0\le\varphi\in C_c^\infty(Q_T)$, and integrate by parts.
Since $\varphi$ vanishes near the time endpoints,
there are no initial or terminal terms, and we obtain
\[
\begin{aligned}
 \iint\varphi\mathbf1_{\{W_j<\delta\}}|\nabla W_j|^2
 \le{}&
 \iint e_{j,\delta}(\rho_j)\partial_t\varphi
 -\iint e_{j,\delta}(\rho_j)
              \nabla W_j\cdot\nabla\varphi\\
 &+\iint z_\delta(W_j)\Delta\varphi\\
 &+\iint\varphi
       [e_{j,\delta}(\rho_j)+\min\{p_j,\delta\}]A_j.
\end{aligned}
\]
All integrals here are over $Q_T$.

Finally, the bounds
\[
 0\le e_{j,\delta}(\rho_j)\le\varepsilon_{j,\delta},
 \qquad
 0\le z_\delta(W_j)\le\delta W_j\le\delta P,
 \qquad
 |A_j|\le B
\]
and $\|\nabla W_j\|_{L^2(Q_T)}\le C_D^{1/2}$ give
\[
\begin{aligned}
 \iint\varphi\mathbf1_{\{W_j<\delta\}}|\nabla W_j|^2
 \le{}
 \varepsilon_{j,\delta}
 \left(
 \|\partial_t\varphi\|_1
 +C_D^{1/2}\|\nabla\varphi\|_2
 +B\|\varphi\|_1
 \right) + \delta\left(
 P\|\Delta\varphi\|_1+B\|\varphi\|_1
 \right).
\end{aligned}
\]
Letting $j\to\infty$ with $\delta$ fixed proves
\[
 \limsup_{j\to\infty}
 \iint\varphi\mathbf1_{\{W_j<\delta\}}|\nabla W_j|^2
 \le C_\varphi\delta,
 \qquad
 C_\varphi=P\|\Delta\varphi\|_1+B\|\varphi\|_1,
\]
where $C_\varphi$ is independent of $j$ and $\delta$.
\end{proof}

\begin{lemma}[Total density flux]\label{lem:density-flux-revised}
For every compact $K\Subset Q_T$,
\[
 \| (\rho_j-1)\nabla W_j\|_{L^1(K)}\longrightarrow0.
\]
Consequently,
\[
 \partial_t\rho=\Delta p+R
       =\operatorname{div}(\rho\nabla p)+R
 \quad\hbox{in distributions}.
\]
\end{lemma}

\begin{proof}
Fix a compact set $K\Subset Q_T$, choose
$0\le\varphi\in C_c^\infty(Q_T)$ with $\varphi=1$ on $K$,
and let $0<2\delta<P$.
Recall the uniform estimates
\[
 |\rho_j-1|\le C_\rho+1=:C_0,\qquad
 \|\nabla W_j\|_{L^2(Q_T)}^2\le C_D,\qquad
 \|W_j-p_j\|_{L^2(Q_T)}^2\le C_D\nu_j.
\]
We decompose $K$ into the three disjoint sets
\[
\begin{aligned}
 L_{j,\delta}&=K\cap\{W_j<2\delta\},\\
 H_{j,\delta}&=K\cap\{W_j\ge2\delta,\ p_j\ge\delta\},\\
 E_{j,\delta}&=K\cap\{W_j\ge2\delta,\ p_j<\delta\}.
\end{aligned}
\]

On $L_{j,\delta}$, Cauchy--Schwarz gives
\[
 \int_{L_{j,\delta}}|\rho_j-1||\nabla W_j|
 \le C_0|K|^{1/2}
 \left(
 \iint_{Q_T}\varphi\mathbf1_{\{W_j<2\delta\}}
                  |\nabla W_j|^2
 \right)^{1/2}.
\]
Lemma~\ref{lem:small-potential-revised}, applied with
threshold $2\delta$, therefore implies
\[
 \limsup_{j\to\infty}
 \int_{L_{j,\delta}}|\rho_j-1||\nabla W_j|
 \le C_0|K|^{1/2}(2C_\varphi\delta)^{1/2}.
\]

On $H_{j,\delta}$, the pressure lies in $[\delta,P]$.
The inverse pressure law reads
\[
 \rho_j=f_j(p_j),\qquad
 f_j(q)=\left(\frac{k_j-1}{k_j}q\right)^{1/(k_j-1)}.
\]
For fixed $\delta$, the estimate
\[
 \sup_{\delta\le q\le P}|\log f_j(q)|
 \le
 \frac{
 |\log((k_j-1)/k_j)|
 +\max\{|\log\delta|,|\log P|\}}
 {k_j-1}
 \longrightarrow0
\]
shows that
\[
 \eta_{j,\delta}
 :=\sup_{\delta\le q\le P}|f_j(q)-1|
 \longrightarrow0.
\]
Consequently,
\[
 \int_{H_{j,\delta}}|\rho_j-1||\nabla W_j|
 \le\eta_{j,\delta}|K|^{1/2}C_D^{1/2}
 \longrightarrow0.
\]

On the remaining set $E_{j,\delta}$, we have
$W_j-p_j\ge\delta$, so
\[
 \delta^2|E_{j,\delta}|
 \le\int_{E_{j,\delta}}|W_j-p_j|^2
 \le C_D\nu_j.
\]
Thus
\[
\begin{aligned}
 \int_{E_{j,\delta}}|\rho_j-1||\nabla W_j|
 \le C_0|E_{j,\delta}|^{1/2}
             \|\nabla W_j\|_{L^2(Q_T)} \le\frac{C_0C_D\sqrt{\nu_j}}{\delta}
 \longrightarrow0
\end{aligned}
\]
with $\delta$ fixed. Combining these estimates gives
\[
 \limsup_{j\to\infty}
 \|(\rho_j-1)\nabla W_j\|_{L^1(K)}
 \le C_0|K|^{1/2}(2C_\varphi\delta)^{1/2}.
\]
Letting $\delta\downarrow0$ proves
\[
 \|(\rho_j-1)\nabla W_j\|_{L^1(K)}
 \longrightarrow0.
\]

Now fix $\psi\in C_c^\infty(Q_T)$.
The total-density equation gives
\[
\begin{aligned}
 -\iint_{Q_T}\rho_j\partial_t\psi
 ={}&-\iint_{Q_T}\nabla W_j\cdot\nabla\psi
     +\iint_{Q_T}R_j\psi\\
 &-\iint_{Q_T}(\rho_j-1)\nabla W_j\cdot\nabla\psi.
\end{aligned}
\]
The last integral tends to zero by the flux estimate on
a compact set containing the support of $\psi$.
Passing to the limit in the other terms using the
previously established weak convergences yields
\[
 -\iint_{Q_T}\rho\partial_t\psi
 =-\iint_{Q_T}\nabla p\cdot\nabla\psi
   +\iint_{Q_T}R\psi.
\]
This is precisely
$\partial_t\rho=\Delta p+R$ in the sense of distributions.

Finally, $p(1-\rho)=0$ implies $\rho=1$ almost everywhere
on $\{p>0\}$. Since $p\ge0$ and
$p\in L^2(0,T;H^1(\mathbb R^d))$, the Sobolev level-set
property gives $\nabla p=0$ almost everywhere on $\{p=0\}$.
Together, these facts imply
\[
 \rho\nabla p=\nabla p\quad\text{a.e. in }Q_T.
\]
The preceding integral identity therefore also reads
\[
 -\iint_{Q_T}\rho\partial_t\psi
 =-\iint_{Q_T}\rho\nabla p\cdot\nabla\psi
   +\iint_{Q_T}R\psi,
\]
which proves
$\partial_t\rho=\operatorname{div}(\rho\nabla p)+R$
in the sense of distributions.
\end{proof}

\subsection{A pressure identity for the weak limits}

\begin{lemma}[pressure energy identity]\label{lem:chain-rule-revised}
For every $\eta\in C_c^\infty(Q_T)$,
\[
 \int_0^T\langle\partial_t\rho,\eta p\rangle\,dt=0,
\]
where the pairing is local $H^{-1}$--$H_0^1$ duality. Consequently,
\begin{equation}\label{eq:weak-complementarity-revised}
 \iint\eta|\nabla p|^2+\iint p\nabla p\cdot\nabla\eta
       =\iint\eta pR.
\end{equation}
In particular, for every $\chi\in C_c^\infty(0,T)$,
\begin{equation}\label{eq:global-pressure-energy-revised}
 \iint\chi(t)|\nabla p|^2=\iint\chi(t)pR.
\end{equation}
This identity also holds for $\chi\equiv1$, by approximation with
temporal cutoffs and dominated convergence.
\end{lemma}

\begin{proof}
For every nonnegative $\zeta\in C_c^\infty(Q_T)$,
the weak convergences of $\rho_j$ and $R_j$ give
\[
 \iint_{Q_T}(B\rho\pm R)\zeta
 =\lim_{j\to\infty}\iint_{Q_T}(B\rho_j\pm R_j)\zeta
 \ge0.
\]
Hence $|R|\le B\rho$ almost everywhere. In particular,
$R\in L^2(Q_T)$ and
\[
 \iint_{Q_T}|pR|
 \le PB\iint_{Q_T}\rho<\infty.
\]
Lemma~\ref{lem:density-flux-revised} gives
$\partial_t\rho=\Delta p+R$ in the sense of distributions.
On every bounded smooth domain $\Omega$, its right-hand
side belongs to $L^2(0,T;H^{-1}(\Omega))$, with
\[
 \|\partial_t\rho(t)\|_{H^{-1}(\Omega)}
 \le\|\nabla p(t)\|_{L^2(\Omega)}
    +\|R(t)\|_{L^2(\Omega)}.
\]
Since $\rho\in L^2(Q_T)$, it follows that
\[
 \rho\in H^1(0,T;H^{-1}(\Omega)).
\]

Let $0\le\eta\in C_c^\infty(Q_T)$ and choose
a bounded smooth domain $\Omega$ and times
$0<t_0<t_1<T$ such that
\[
 \operatorname{supp}\eta\Subset\Omega\times(t_0,t_1).
\]
The identities
\[
 \eta p\in L^2(0,T;L^2(\Omega)),\qquad
 \nabla(\eta p)=\eta\nabla p+p\nabla\eta
\]
and the compact spatial support of $\eta$ imply
\[
 \eta p\in L^2(0,T;H_0^1(\Omega)).
\]

For $0<|h|<\min\{t_0,T-t_1\}$, the fundamental theorem
of calculus for $H^{-1}(\Omega)$-valued Sobolev functions
gives
\[
 \frac{\rho(t+h)-\rho(t)}{h}
 =\int_0^1\partial_t\rho(t+sh)\,ds
 \quad\text{for almost every }t\in(t_0,t_1).
\]
Jensen's inequality therefore yields
\[
\begin{aligned}
 &\int_{t_0}^{t_1}
 \left\|
 \frac{\rho(t+h)-\rho(t)}{h}-\partial_t\rho(t)
 \right\|_{H^{-1}(\Omega)}^2\,dt\\
 &\qquad\le
 \int_0^1\int_{t_0}^{t_1}
 \|\partial_t\rho(t+sh)-\partial_t\rho(t)\|_{H^{-1}(\Omega)}^2
 \,dt\,ds
 \longrightarrow0
 \qquad\text{as }h\to0.
\end{aligned}
\]
Here the last limit follows from continuity of time
translations in $L^2(\mathbb R;H^{-1}(\Omega))$,
after extending $\partial_t\rho$ by zero outside $(0,T)$.

The relations $0\le\rho\le1$, $p\ge0$, and $p\rho=p$
give, for each such $h$ and almost every $t$,
\[
\begin{aligned}
 &\int_\Omega\eta(x,t)p(x,t)
       [\rho(x,t+h)-\rho(x,t)]\,dx\\
 &\qquad=
 \int_\Omega\eta(x,t)p(x,t)[\rho(x,t+h)-1]\,dx
 \le0.
\end{aligned}
\]
Dividing by $h$ and integrating in time gives
\[
 \int_{t_0}^{t_1}\int_\Omega
 \eta(x,t)p(x,t)
 \frac{\rho(x,t+h)-\rho(x,t)}{h}\,dx\,dt
 \begin{cases}
 \le0,&h>0,\\
 \ge0,&h<0.
 \end{cases}
\]
These spatial integrals also equal the corresponding
$H^{-1}$--$H_0^1$ duality pairings, since
$[\rho(t+h)-\rho(t)]/h\in L^2(\Omega)$
for almost every $t$.

Moreover,
\[
\begin{aligned}
 &\left|
 \int_{t_0}^{t_1}\int_\Omega
 \eta p\,\frac{\rho(t+h)-\rho(t)}{h}\,dx\,dt
 -\int_{t_0}^{t_1}
       \langle\partial_t\rho,\eta p\rangle\,dt
 \right|\\
 &\qquad\le
 \left(
 \int_{t_0}^{t_1}
 \left\|
 \frac{\rho(t+h)-\rho(t)}{h}-\partial_t\rho(t)
 \right\|_{H^{-1}(\Omega)}^2\,dt
 \right)^{1/2}
 \|\eta p\|_{L^2(t_0,t_1;H_0^1(\Omega))}
 \longrightarrow0.
\end{aligned}
\]
Taking $h\to0$ through positive values consequently gives
\[
 \int_{t_0}^{t_1}
 \langle\partial_t\rho,\eta p\rangle\,dt\le0,
\]
whereas taking $h\to0$ through negative values gives
\[
 \int_{t_0}^{t_1}
 \langle\partial_t\rho,\eta p\rangle\,dt\ge0.
\]
Since $\eta$ vanishes outside $(t_0,t_1)$, we conclude that
\[
 \int_0^T\langle\partial_t\rho,\eta p\rangle\,dt=0.
\]

For an arbitrary $\eta\in C_c^\infty(Q_T)$, choose
$0\le\xi\in C_c^\infty(Q_T)$ equal to one on
$\operatorname{supp}\eta$. Both
$\eta+\|\eta\|_\infty\xi$ and $\|\eta\|_\infty\xi$
are nonnegative. Applying the preceding identity to these
two functions and subtracting proves the same identity
for $\eta$.

The density equation holds as an equality in
$H^{-1}(\Omega)$ for almost every time. Pairing it with
$\eta p$ and expanding $\nabla(\eta p)$ therefore gives
\[
\begin{aligned}
 0
 &=\int_0^T\langle\partial_t\rho,\eta p\rangle\,dt\\
 &=-\iint_{Q_T}\nabla p\cdot\nabla(\eta p)
   +\iint_{Q_T}\eta pR\\
 &=-\iint_{Q_T}\eta|\nabla p|^2
   -\iint_{Q_T}p\nabla p\cdot\nabla\eta
   +\iint_{Q_T}\eta pR.
\end{aligned}
\]
Rearranging proves
\eqref{eq:weak-complementarity-revised}.

Finally, choose $\chi\in C_c^\infty(0,T)$ and
$\theta\in C_c^\infty(\mathbb R^d)$ with
$0\le\theta\le1$ and $\theta=1$ on $B_1$.
Substituting $\eta(x,t)=\chi(t)\theta(x/L)$ gives
\[
\begin{aligned}
 &\iint_{Q_T}\chi(t)\theta(x/L)|\nabla p|^2\\
 &\qquad+
 \iint_{Q_T}\chi(t)p\nabla p\cdot\nabla[\theta(x/L)]
 =\iint_{Q_T}\chi(t)\theta(x/L)pR.
\end{aligned}
\]
The term containing the spatial derivative of the cutoff
satisfies
\[
 \left|
 \iint_{Q_T}\chi(t)p\nabla p\cdot\nabla[\theta(x/L)]
 \right|
 \le
 \frac{\|\chi\|_\infty\|\nabla\theta\|_\infty}{L}
 \|p\|_{L^2(Q_T)}\|\nabla p\|_{L^2(Q_T)}
 \longrightarrow0.
\]
Since $\theta(x/L)\to1$ pointwise and
$|\nabla p|^2,|pR|\in L^1(Q_T)$, dominated convergence
in the other two terms yields
\[
 \iint_{Q_T}\chi(t)|\nabla p|^2
 =\iint_{Q_T}\chi(t)pR,
\]
which proves \eqref{eq:global-pressure-energy-revised}.

Choose $0\le\chi_m\le1$ in $C_c^\infty(0,T)$ with
$\chi_m(t)\to1$ for every $t\in(0,T)$.
Applying the last identity with $\chi=\chi_m$ and using
dominated convergence once more gives
\[
 \iint_{Q_T}|\nabla p|^2=\iint_{Q_T}pR.
\]
\end{proof}

\section{Growth limsup and disappearance of the defect}\label{sec:strong-identification}
\subsection{The growth limsup without a sign restriction}

We keep the arbitrary simultaneous sequence
$(N_j,k_j,\nu_j)\to(\infty,\infty,0)$ fixed throughout.
Write $R_j\rightharpoonup R$ in $L^2(Q_T)$; the identification of $R$
will be made only after pressure compactness has been proved. The constants $P,B,L_a$ have been fixed in Section~\ref{sec:estimates}; $B$ bounds the absolute value of $G$. No nonnegativity assumption on $G$ is made.

\begin{lemma}[Growth limsup]\label{lem:growth-limsup-repaired}
For every nonnegative $\eta\in C^1([0,T])$,
\begin{equation}\label{eq:growth-limsup-repaired}
 \limsup_{j\to\infty}\left\{
 \iint_{Q_T}\eta R_j(p_j-p)
 +c\iint_{Q_T}\eta\rho_j|p_j-p|^2\right\}\leq0.
\end{equation}
In particular,
\begin{equation}\label{eq:growth-limsup-energy}
 \limsup_{j\to\infty}\left\{
 \iint_{Q_T}\eta R_jp_j
 +c\iint_{Q_T}\eta\rho_j|p_j-p|^2\right\}
 \leq\iint_{Q_T}\eta Rp.
\end{equation}
\end{lemma}

\begin{proof}
Fix $0\le\eta\in C^1([0,T])$ and write
\[
 B_j(p)(x,t)
 :=\int_Iu_j(x,t,a)G(p(x,t),a)\,da.
\]
We first prove
\begin{equation}\label{eq:frozen-growth-product}
 \iint_{Q_T}\eta B_j(p)(p_j-p)\,dx\,dt
 \longrightarrow0.
\end{equation}

Fix $L>0$ and $0<\tau<T/2$, and set
\[
 D_{L,\tau}=B_L\times(\tau,T-\tau).
\]
The functions $\eta(t)G(p(x,t),a)$ and
$\eta(t)p(x,t)G(p(x,t),a)$ are independent of $j$ and
bounded on $D_{L,\tau}\times I$. Since this set has finite
measure, both functions belong to $L^2(D_{L,\tau}\times I)$.
Lemma~\ref{lem:weak-products-revised} and the weak
convergence of $u_j$ therefore give, respectively,
\[
\begin{aligned}
 \int_{D_{L,\tau}\times I}\eta G(p,a)u_jp_j
 &\longrightarrow
 \int_{D_{L,\tau}\times I}\eta G(p,a)np,\\
 \int_{D_{L,\tau}\times I}\eta pG(p,a)u_j
 &\longrightarrow
 \int_{D_{L,\tau}\times I}\eta pG(p,a)n.
\end{aligned}
\]
Subtracting these limits yields
\[
 \int_{D_{L,\tau}}\eta B_j(p)(p_j-p)\,dx\,dt
 \longrightarrow0.
\]

To pass from this cylinder to $Q_T$, recall the uniform
mass and first-moment bounds
\[
 \sup_{j,t}\int_{\mathbb R^d}\rho_j\,dx\le M,
 \qquad
 \sup_{j,t}\int_{\mathbb R^d}|x|\rho_j\,dx\le C_M.
\]
Since $|B_j(p)|\le B\rho_j$ and $|p_j-p|\le P$,
\[
 |\eta B_j(p)(p_j-p)|
 \le PB\|\eta\|_\infty\rho_j.
\]
Moreover,
\[
\begin{aligned}
 \int_{Q_T\setminus D_{L,\tau}}\rho_j\,dx\,dt
 &\le
 \int_0^T\int_{|x|\ge L}\rho_j\,dx\,dt\\
 &\quad+
 \left(\int_0^\tau+\int_{T-\tau}^T\right)
       \int_{\mathbb R^d}\rho_j\,dx\,dt\\
 &\le \frac{TC_M}{L}+2M\tau.
\end{aligned}
\]
Consequently,
\[
 \limsup_{j\to\infty}
 \left|\iint_{Q_T}\eta B_j(p)(p_j-p)\right|
 \le PB\|\eta\|_\infty
       \left(\frac{TC_M}{L}+2M\tau\right).
\]
Letting $L\to\infty$ and $\tau\downarrow0$ proves
\eqref{eq:frozen-growth-product}.

We now use the strict decrease of $G$ in its pressure
argument. For almost every $(x,t,a)$,
\[
\begin{aligned}
 &[G(p_j,a)-G(p,a)](p_j-p)\\
 &\qquad=
 |p_j-p|^2\int_0^1
 \partial_qG\bigl(p+s(p_j-p),a\bigr)\,ds
 \le-c|p_j-p|^2.
\end{aligned}
\]
The reconstructed growth coefficient also satisfies,
for $a\in I_{j,i}$,
\[
 |G_j(p_j,a)-G(p_j,a)|
 =|G(p_j,a_{j,i})-G(p_j,a)|
 \le\frac{L_a}{N_j}.
\]
Define the resulting error by
\[
 E_j
 :=(p_j-p)\int_Iu_j
       [G_j(p_j,a)-G(p_j,a)]\,da.
\]
Then
\[
 |E_j|\le\frac{PL_a}{N_j}\rho_j.
\]
Using the definition of $R_j$ and adding and subtracting
$G(p,a)$ gives the exact decomposition
\[
\begin{aligned}
 R_j(p_j-p)
 ={}&\int_Iu_j[G(p_j,a)-G(p,a)](p_j-p)\,da\\
 &+B_j(p)(p_j-p)+E_j.
\end{aligned}
\]
Since $u_j\ge0$ and $\int_Iu_j\,da=\rho_j$, we obtain
\[
 R_j(p_j-p)+c\rho_j|p_j-p|^2
 \le B_j(p)(p_j-p)+E_j.
\]
Multiplying by $\eta\ge0$ and integrating yields
\[
\begin{aligned}
 &\iint_{Q_T}\eta R_j(p_j-p)
   +c\iint_{Q_T}\eta\rho_j|p_j-p|^2\\
 &\qquad\le
 \iint_{Q_T}\eta B_j(p)(p_j-p)
 +\frac{PL_aTM\|\eta\|_\infty}{N_j}.
\end{aligned}
\]
The right-hand side tends to zero by
\eqref{eq:frozen-growth-product} and $N_j\to\infty$.
This proves \eqref{eq:growth-limsup-repaired}.

Finally, $\eta p$ is a fixed function in $L^2(Q_T)$,
so the weak convergence of $R_j$ gives
\[
 \iint_{Q_T}\eta R_jp\longrightarrow
 \iint_{Q_T}\eta Rp.
\]
Adding this convergent term to
\eqref{eq:growth-limsup-repaired}, and using
$R_jp_j=R_j(p_j-p)+R_jp$, proves
\eqref{eq:growth-limsup-energy}.
\end{proof}

\subsection{Saturation of the pressure dissipation}

At this stage, $R$ denotes only the weak $L^2(Q_T)$ limit
of $R_j$. The relations
$\partial_t\rho=\Delta p+R$ and $p(1-\rho)=0$ have already
been established. Lemma~\ref{lem:chain-rule-revised}
therefore gives:

\[
 \int_0^T\chi(t)\int_{\mathbb R^d}|\nabla p|^2\,dx\,dt
 =
 \int_0^T\chi(t)\int_{\mathbb R^d}pR\,dx\,dt
 \qquad
 \text{for every }\chi\in C_c^\infty(0,T).
\]

Since $|\nabla p|^2$ and $|pR|$ belong to $L^1(Q_T)$,
this implies:

\[
 \int_{\mathbb R^d}|\nabla p(x,t)|^2\,dx
 =
 \int_{\mathbb R^d}p(x,t)R(x,t)\,dx
 \quad\text{for almost every }t\in(0,T).
\]

Integrating in time yields:

\begin{equation}\label{eq:independent-limit-energy-used}
 \iint_{Q_T}|\nabla p|^2\,dx\,dt
 =
 \iint_{Q_T}pR\,dx\,dt.
\end{equation}

Information about $p(x,T)$ is not needed here, since no integration
by parts in time is involved in this last passage.
Thus the energy identity follows from the limiting
density equation and the density--pressure constraints,
before identifying $R$ or establishing strong pressure
convergence.

\begin{proposition}[convergence of dissipation]\label{prop:full-dissipation}
Along the fixed simultaneous sequence,
\begin{align}
 \nabla W_j&\longrightarrow\nabla p
       &&\text{strongly in }L^2(Q_T),\label{eq:strong-velocity-global}\\
 \nu_j|\Delta W_j|^2&\longrightarrow0
       &&\text{strongly in }L^1(Q_T),\label{eq:no-brinkman-defect}\\
 \iint_{Q_T}\rho_j|p_j-p|^2&\longrightarrow0.
       \label{eq:weighted-pressure-variance}
\end{align}
\end{proposition}

\begin{proof}
Lemma~\ref{lem:pressure-dissipation}, after dropping the
nonnegative terminal pressure energy, gives
\begin{equation}\label{eq:full-pressure-dissipation-upper}
 \iint_{Q_T}
 \bigl(|\nabla W_j|^2+\nu_j|\Delta W_j|^2\bigr)
 \le
 \frac{k_j-1}{k_j-2}\iint_{Q_T}p_jA_j
 +\frac{1}{k_j-2}
       \int_{\mathbb R^d}p_j^{\rm in}\,dx.
\end{equation}
The initial pressure mass is uniformly bounded by
$C_pM_0$.

We first justify replacing $p_jA_j$ by $p_jR_j$
over the whole of $Q_T$.
For a scalar $r\in[0,1]$, direct maximization gives
\[
 \max_{0\le r\le1}r^{k_j-2}(1-r)
 =
 \frac{1}{k_j-1}
 \left(\frac{k_j-2}{k_j-1}\right)^{k_j-2}
 \le\frac{1}{k_j-1}.
\]
Thus, on $\{0\le\rho_j\le1\}$, the pressure law implies
\[
 p_j|1-\rho_j|
 =\frac{k_j}{k_j-1}\rho_j
       \rho_j^{k_j-2}(1-\rho_j)
 \le\frac{3\rho_j}{2(k_j-1)}.
\]
On $\{\rho_j\ge1\}$, the inverse pressure law gives
\[
 0\le\log\rho_j
 =\frac{\log((k_j-1)/k_j)+\log p_j}{k_j-1}
 \le\frac{\log(\max\{1,P\})}{k_j-1}.
\]
Using $\rho_j-1\le\rho_j\log\rho_j$,
$\rho_j\le C_\rho$, and $p_j\le C_p\rho_j$, we obtain
\[
 p_j|\rho_j-1|
 \le
 \frac{C_pC_\rho\log(\max\{1,P\})}{k_j-1}\rho_j.
\]
Combining the two regions proves
\begin{equation}\label{eq:pressure-graph-global-error}
 p_j|1-\rho_j|
 \le\frac{C}{k_j-1}\rho_j,
\end{equation}
with $C$ depending only on $P$.

Since $R_j=\rho_jA_j$ and $|A_j|\le B$, this yields
\[
\begin{aligned}
 \left|\iint_{Q_T}p_jA_j-\iint_{Q_T}p_jR_j\right| \le B\iint_{Q_T}p_j|1-\rho_j| \le\frac{CB}{k_j-1}\|\rho_j\|_{L^1(Q_T)}
 \longrightarrow0.
\end{aligned}
\]
The factor $(k_j-1)/(k_j-2)$ can also be removed with
an explicitly vanishing error:
\[
 \left|
 \frac{k_j-1}{k_j-2}\iint_{Q_T}p_jA_j
 -\iint_{Q_T}p_jA_j
 \right|
 \le\frac{B}{k_j-2}\|p_j\|_{L^1(Q_T)}.
\]
These estimates use absolute values and require no sign
assumption on $A_j$ or $R_j$.

Substituting into
\eqref{eq:full-pressure-dissipation-upper} gives
\[
\begin{aligned}
 &\iint_{Q_T}
       \bigl(|\nabla W_j|^2+\nu_j|\Delta W_j|^2\bigr)\\
 &\qquad\le
 \iint_{Q_T}p_jR_j
 +\frac{B}{k_j-2}\|p_j\|_{L^1(Q_T)}
 +\frac{CB}{k_j-1}\|\rho_j\|_{L^1(Q_T)}
 +\frac{C_pM_0}{k_j-2}.
\end{aligned}
\]
The last three terms tend to zero by the uniform mass
bounds and $k_j\to\infty$. Adding
$c\iint_{Q_T}\rho_j|p_j-p|^2$ and applying
\eqref{eq:growth-limsup-energy} with $\eta\equiv1$
therefore yields
\begin{equation}\label{eq:combined-saturation}
\begin{aligned}
 &\limsup_{j\to\infty}\left[
 \iint_{Q_T}
       \bigl(|\nabla W_j|^2+\nu_j|\Delta W_j|^2\bigr)
 +c\iint_{Q_T}\rho_j|p_j-p|^2
 \right]\\
 &\qquad\le
 \iint_{Q_T}Rp
 =\iint_{Q_T}|\nabla p|^2.
\end{aligned}
\end{equation}
The equality here is
\eqref{eq:independent-limit-energy-used}, established
before strong pressure convergence and before identifying
the weak reaction limit $R$.

To extract the consequences of this upper bound, use the
previously established weak convergence of $\nabla W_j$
in $L^2(Q_T)$. Testing it against the fixed function
$\nabla p$ gives
\[
 \iint_{Q_T}\nabla W_j\cdot\nabla p
 \longrightarrow\iint_{Q_T}|\nabla p|^2.
\]
The exact identity
\[
\begin{aligned}
 &\iint_{Q_T}|\nabla W_j-\nabla p|^2
 +\nu_j\iint_{Q_T}|\Delta W_j|^2
 +c\iint_{Q_T}\rho_j|p_j-p|^2\\
 &\quad=
 \iint_{Q_T}
       \bigl(|\nabla W_j|^2+\nu_j|\Delta W_j|^2\bigr)
 +c\iint_{Q_T}\rho_j|p_j-p|^2\\
 &\qquad\quad
 +\iint_{Q_T}|\nabla p|^2
 -2\iint_{Q_T}\nabla W_j\cdot\nabla p
\end{aligned}
\]
and \eqref{eq:combined-saturation} consequently imply
\[
\begin{aligned}
 \limsup_{j\to\infty}\left[
 \iint_{Q_T}|\nabla W_j-\nabla p|^2
 +\nu_j\iint_{Q_T}|\Delta W_j|^2
 +c\iint_{Q_T}\rho_j|p_j-p|^2
 \right]
 \le0.
\end{aligned}
\]
Each of the three terms inside the brackets is
nonnegative. Their sum therefore tends to zero,
and each term tends to zero separately. In particular,
\[
 \bigl\|\nu_j|\Delta W_j|^2\bigr\|_{L^1(Q_T)}
 =\nu_j\iint_{Q_T}|\Delta W_j|^2
 \longrightarrow0.
\]
This conclusion follows from the comparison of the full
dissipation with the independently established limiting
energy identity; the uniform dissipation bound alone
would not imply it.

The other two terms give
\[
 \|\nabla W_j-\nabla p\|_{L^2(Q_T)}^2\longrightarrow0,
 \qquad
 \iint_{Q_T}\rho_j|p_j-p|^2\longrightarrow0,
\]
completing the proof.
\end{proof}

\begin{corollary}[Strong pressure convergence]\label{cor:strong-pressure-repaired}
Along the same sequence,
\[
 p_j\longrightarrow p,\qquad W_j\longrightarrow p
 \quad\text{strongly in }L^2(Q_T).
\]
\end{corollary}

\begin{proof}
Let $f_j=(1-\rho_j)_+$. Pointwise,
\[
 |p_j-p|^2
 \leq\rho_j|p_j-p|^2+2f_jp_j^2+2f_jp^2.
\]
For $0\leq r\leq1$,
\[
 (1-r)P_j(r)^2
 \leq\frac{C}{2k_j-2}r,
\]
because $\sup_{0\leq r\leq1}(1-r)r^{2k-3}\leq1/(2k-2)$ for any $k$.
Thus $\iint f_jp_j^2\to0$. Also
\[
 0\leq\rho_j\leq
 \left(\frac{k_j-1}{k_j}P\right)^{1/(k_j-1)}
 \leq1+\varepsilon_j,\qquad\varepsilon_j\to0.
\]
Since $\rho_j\rightharpoonup^*\rho$ in $L^\infty(Q_T)$ and
$p^2\in L^1(Q_T)$, the pressure-density graph gives
\[
 \iint f_jp^2
 =\iint(1-\rho_j)p^2+\iint(\rho_j-1)_+p^2
 \longrightarrow\iint(1-\rho)p^2=0.
\]
Use \eqref{eq:weighted-pressure-variance} to conclude that
$p_j\to p$ in $L^2(Q_T)$. The estimate
$\|W_j-p_j\|_{L^2(Q_T)}\leq C\sqrt{\nu_j}$ then proves the
assertion for $W_j$.
\end{proof}

\begin{remark}
Oscillations and concentrations in weakly convergent sequences have a broader history in fluid equations; see \cite{DiPerma1987} for this background. Here the preceding dissipation argument rules out the Brinkman concentration term.
In particular, if the nonnegative measures
$\nu_j|\Delta W_j|^2\,dx\,dt$ have a weak limit, that limit is
the zero measure. No relative rate between $N_j$, $k_j$, and
$\nu_j$ was imposed at any point of the argument.
\end{remark}

\section{Reaction identification and completion of the proof}
\label{sec:reaction-with-bv}
\begin{proposition}[Reaction identity]\label{prop:reaction-with-bv}
Suppose $u_j\rightharpoonup n$ in $L^2(Q_T\times I)$ and
$p_j\to p$ strongly in $L^2_{\mathrm{loc}}(Q_T)$. Then
\begin{equation}\label{eq:bv-reaction-convergence}
 u_jG_j(p_j,a)\rightharpoonup nG(p,a)
       \quad\text{in }L^2(Q_T\times(0,1)),
 \qquad
 R_j\rightharpoonup R:=\int_0^1nG(p,a)\,da
       \quad\text{in }L^2(Q_T).
\end{equation}
The reaction densities also satisfy $|R|\leq B\rho$.
If $\nabla W_j\to\nabla p$ strongly in $L^2_{\mathrm{loc}}(Q_T)$,
then
\[
 \partial_tn=\operatorname{div}_x(n\nabla p)+nG(p,a),\qquad
 \partial_t\rho=\operatorname{div}_x(\rho\nabla p)
                   +\int_0^1nG(p,a)\,da
\]
in distributions. The initial terms pass whenever the reconstructed
initial densities converge weakly in $L^1(\mathbb R^d\times(0,1))$
to the prescribed $n^{\mathrm{in}}$.
\end{proposition}
\begin{proof}
Let $K$ be a compact space--time cylinder. By the definition of the step reconstruction,
\[
 |G_j(p_j,a)-G(p,a)|\leq L_p|p_j-p|+L_a/N_j.
\]
Hence
\[
 \|G_j(p_j,a)-G(p,a)\|_{L^2(K\times(0,1))}
 \leq L_p\|p_j-p\|_{L^2(K)}+L_a|K|^{1/2}/N_j
 \longrightarrow0.
\]
For a bounded test $\Phi$ supported in $K\times(0,1)$, decompose
\begin{align*}
 \int\Phi\bigl[u_jG_j(p_j,a)-nG(p,a)\bigr]
 ={}&\int\Phi u_j\bigl[G_j(p_j,a)-G(p,a)\bigr]\\
    &+\int\Phi(u_j-n)G(p,a).
\end{align*}
The first term tends to zero by Cauchy--Schwarz and the uniform joint
$L^2$ estimate. The second tends to zero by weak $L^2$ convergence
against the fixed test $\Phi G(p,a)$. The products on the left are
uniformly bounded in global $L^2$ because $|G_j|\leq B$. Approximation
of arbitrary $L^2$ tests by bounded tests with compact support therefore
proves the first convergence in~\eqref{eq:bv-reaction-convergence}.
Integrating in $a$ proves the second. The bound $|R_j|\leq B\rho_j$
passes to the weak limit and gives $|R|\leq B\rho$.

The exact step reconstruction satisfies, for every smooth compactly
supported $\Phi(x,t,a)$ with $\Phi(\cdot,T,\cdot)=0$,
\begin{align*}
 0={}&\int_{Q_T\times(0,1)}u_j\partial_t\Phi
 -\int_{Q_T\times(0,1)}u_j\nabla W_j\cdot\nabla_x\Phi\\
 &+\int_{Q_T\times(0,1)}u_jG_j(p_j,a)\Phi
 +\int_{\mathbb R^d\times(0,1)}u_j^{\mathrm{in}}(x,a)\Phi(x,0,a).
\end{align*}
Indeed, test the $i$th equation with
$N_j\int_{I_{j,i}}\Phi(x,t,a)\,da$, multiply by $N_j^{-1}$, and
sum over $i$. This step is an exact identity, not a Riemann-sum passage.
For the flux, on compact subsets of $Q_T\times(0,1)$,
\[
 u_j\nabla W_j-n\nabla p
 =u_j(\nabla W_j-\nabla p)+(u_j-n)\nabla p.
\]
The first term tends to zero in $L^1$ by Cauchy--Schwarz; the second
converges weakly against bounded smooth tests because $\nabla p$
belongs to $L^2$ and $u_j\rightharpoonup n$ in the joint $L^2$ space.
If only interior strong velocity convergence is used, the flux omitted
from $(0,\delta)$ is bounded by
\[
 C\|u_j\|_{L^2(\mathbb R^d\times(0,\delta)\times(0,1))}
       \|\nabla W_j\|_{L^2(Q_T)}\leq C\sqrt\delta;
\]
this follows from~\eqref{eq:trait-joint-quadratic-bound}. Thus tests
meeting the initial time are also admissible after letting
$\delta\downarrow0$. The reaction terms pass by the first part of the
proof, and the initial terms pass by their assumed weak $L^1$
convergence. This proves the phenotype equation with its initial trace.
Integrating in $a$ proves the total-density equation. Fubini and a
countable dense family of space--time tests give the phenotype equation
for almost every fixed $a$ if desired. Every convergence uses the same
subsequence of the initially prescribed simultaneous sequence.
\end{proof}

\begin{proof}[Completion of the proof of Theorem~\ref{thm:main}]
Start with an arbitrary simultaneous sequence
\eqref{seq}. The uniform bounds and the extraction
in \eqref{weak-limits} give, along a subsequence,
\[
\begin{aligned}
 u_j&\rightharpoonup n
     &&\text{in }L^2(Q_T\times I),\\
 \rho_j&\rightharpoonup\rho,\qquad
 p_j\rightharpoonup p,\qquad
 R_j\rightharpoonup R
     &&\text{in }L^2(Q_T),\\
 W_j&\rightharpoonup p
     &&\text{in }L^2(0,T;H^1(\mathbb R^d)).
\end{aligned}
\]
The limits are nonnegative except for the possibly signed
reaction $R$, and the established bounds imply
\[
 n\in L^\infty\bigl(0,T;
       L^1(\mathbb R^d\times I)\cap L^2(\mathbb R^d\times I)\bigr),
 \qquad
 p\in L^\infty(Q_T)\cap L^2(0,T;H^1(\mathbb R^d)).
\]
Lemma~\ref{lem:congestion-revised} identifies the density
and the density--pressure constraints:
\[
 \rho=\int_I n\,da,\qquad
 0\le\rho\le1,\qquad p\ge0,\qquad p(1-\rho)=0.
\]
Lemma~\ref{lem:density-flux-revised} gives
$\rho\nabla p=\nabla p$ almost everywhere and
$\partial_t\rho=\Delta p+R$ in the sense of distributions.

Proposition~\ref{prop:full-dissipation} and
Corollary~\ref{cor:strong-pressure-repaired} establish
the global limits
\[
\begin{aligned}
 &\|p_j-p\|_{L^2(Q_T)}
 +\|W_j-p\|_{L^2(Q_T)}
 +\|\nabla W_j-\nabla p\|_{L^2(Q_T)}
 \longrightarrow0,\\
 &\|\nu_j|\Delta W_j|^2\|_{L^1(Q_T)}
 \longrightarrow0.
\end{aligned}
\]
Thus the pressure and velocity conclusions of the theorem,
including the vanishing Brinkman dissipation, hold on
the whole of $Q_T$.

Proposition~\ref{prop:reaction-with-bv} now applies and gives
\[
 u_jG_j(p_j,a)\rightharpoonup nG(p,a)
 \quad\text{in }L^2(Q_T\times I).
\]
Consequently, the previously extracted reaction limit is
\[
 R(x,t)=\int_I n(x,t,a)G(p(x,t),a)\,da.
\]

The initial terms are determined by the fixed profile
in Assumption~\ref{ass:data}. For $a\in I_{j,i}$,
\[
 |u_j^{\rm in}(x,a)-n_0(x,a)|
 =|n_0(x,i/N_j)-n_0(x,a)|
 \le\frac{L_{\rm in}}{N_j}\mathbf1_K(x),
\]
where $K$ is the common compact spatial support of $n_0$.
In particular,
\[
 \int_{\mathbb R^d\times I}
       |u_j^{\rm in}-n_0|\,dx\,da
 \le\frac{L_{\rm in}|K|}{N_j}\longrightarrow0.
\]
The passage to the phenotype equation in
Proposition~\ref{prop:reaction-with-bv} therefore includes
the initial term. Explicitly, for every
$\Phi\in C_c^\infty(\mathbb R^d\times[0,T)\times I)$,
\[
\begin{aligned}
 0={}&
 \int_{Q_T\times I}n\,\partial_t\Phi
 -\int_{Q_T\times I}n\nabla p\cdot\nabla_x\Phi\\
 &+\int_{Q_T\times I}nG(p,a)\Phi
 +\int_{\mathbb R^d\times I}n_0(x,a)\Phi(x,0,a).
\end{aligned}
\]
Hence
\[
 \partial_t n=\operatorname{div}_x(n\nabla p)+nG(p,a)
\]
in the sense of distributions, with initial datum
$n^{\rm in}=n_0$. Integrating in $a$ gives
\[
 \partial_t\rho
 =\operatorname{div}_x(\rho\nabla p)
   +\int_I nG(p,a)\,da,
 \qquad
 \rho^{\rm in}(x)=\int_I n_0(x,a)\,da,
\]
with the initial datum understood through the corresponding
weak formulation.

Finally, Lemma~\ref{lem:chain-rule-revised}, together with
the identification $R=\int_I nG(p,a)\,da$, gives
\[
 -\iint_{Q_T}p\nabla p\cdot\nabla\phi
 -\iint_{Q_T}\phi|\nabla p|^2
 +\iint_{Q_T}\phi p\int_I nG(p,a)\,da
 =0
 \qquad
 \bigl(\phi\in C_c^\infty(Q_T)\bigr).
\]
This means that
\[
 \operatorname{div}_x(p\nabla p)-|\nabla_xp|^2
 +p\int_I nG(p,a)\,da=0
\]
in the sense of distributions. It proves the third
equation of \eqref{limit-system} and completes the
identification of the limiting equations.
\end{proof}

\bigskip\noindent\textbf{Acknowledgements.}
The authors thank Michael I. Weinstein for very stimulating discussions.
CL is supported in part by the Simons Foundation Math + X Investigator Award \#376319 (MIW) and AMS-Simons Travel Grant.

\appendix

\section{Proof to Lemma \ref{lem:weak-solution-chain-rules}}\label{sec:lemma2.1-proof}

\begin{proof}
The assumed bounds on $n_j^i$ inside Definition \ref{def:solution} give $\rho_j\in L^\infty_t(L^1_x\cap L^\infty_x)$. The power law therefore establishes the same bounds for $p_j$. Then elliptic regularity and the convolution resolvent yield:
\[
 W_j\in L^\infty_tW^{2,q}_x\quad(1<q<\infty),
 \qquad
 \|\nabla W_j\|_\infty
       \leq C_d\nu_j^{-1/2}\|p_j\|_\infty.
\]
Also,
\[
 \Delta W_j=(W_j-p_j)/\nu_j\in L^\infty(Q_T).
\]
And by using Fourier transform with Plancherel identity, we see:
\[
\| D^2W_j \|_{L^2(Q_T)} \leq \frac{\|p_j\|_{L^2(Q_T)}}{\nu_j} \leq \frac{\sqrt{\|p_j\|_{L^1(Q_T)}\|p_j\|_{L^{\infty}(Q_T)}}}{\nu_j}
\]
So
$$D^2W \in L^2(Q_T)$$

Thus, writing $b=\nabla W_j$,
\[
 b\in L^\infty(Q_T)
       \cap L^1_tW^{1,1}_{x,\rm loc},
 \qquad \operatorname{div}b\in L^\infty(Q_T) \qquad Db \in L^2(Q_T).
\]
The coefficients $G_j^i$ are bounded at fixed $j$, so all
species, total-density, and species-difference source terms
belong to $L^1(Q_T)$.

We now prove the same calculation once for each equation.
In this step, $z$ denotes either $n_j^i$, $\rho_j$, or a
difference of species; $S$ denotes its corresponding source.
Thus
\[
 \partial_t z=\operatorname{div}(zb)+S.
\]
Let $\chi_\varepsilon$ be a standard spatial mollifier and set
$z_\varepsilon=z*\chi_\varepsilon$,
$S_\varepsilon=S*\chi_\varepsilon$.
Then
\[
 \partial_t z_\varepsilon
 =\operatorname{div}(z_\varepsilon b)
       +S_\varepsilon+C_\varepsilon,
\]
where
\begin{align*}
 C_\varepsilon(x,t)
 ={}&
 \int z(x-y,t)[b(x-y,t)-b(x,t)]
                  \cdot\nabla\chi_\varepsilon(y)\,dy\\
 &-z_\varepsilon(x,t)\operatorname{div}b(x,t).
\end{align*}

We first verify $C_\varepsilon\to0$ using $Db\in L^2(Q_T)$. Note that $z\in L^2(Q_T)$ by Holder's inequality. Then changing variables gives:

\[
 C_\varepsilon(x,t)
 =
 \int z(x-\varepsilon y,t)
       \frac{b(x-\varepsilon y,t)-b(x,t)}{\varepsilon}
                  \cdot\nabla\chi(y)\,dy
 -z_\varepsilon(x,t)\operatorname{div}b(x,t).
\]

For each fixed $y$, the Sobolev difference-quotient identity gives:

\[
 \frac{b(x-\varepsilon y,t)-b(x,t)}{\varepsilon}
 =-\int_0^1 Db(x-\theta\varepsilon y,t)y\,d\theta.
\]

Consequently,
\begin{align*}
 &\left\|
 \frac{b(\cdot-\varepsilon y,\cdot)-b}{\varepsilon}
       +Db\,y
 \right\|_{L^2(Q_T)}\\
 &\qquad\leq
 |y|\int_0^1
       \|Db(\cdot-\theta\varepsilon y,\cdot)-Db\|_{L^2(Q_T)}
       \,d\theta
 \longrightarrow0,
\end{align*}
by continuity of translations in $L^2$ proved in \cite{Evans2010}.

The same identity gives the bound
\[
 \left\|
 \frac{b(\cdot-\varepsilon y,\cdot)-b}{\varepsilon}
 \right\|_{L^2(Q_T)}
 \leq |y|\|Db\|_{L^2(Q_T)}.
\]

Since $z(\cdot-\varepsilon y,\cdot)\to z$ in $L^2$ as well,
Cauchy--Schwarz yields
\begin{align*}
 &\left\|
 z(\cdot-\varepsilon y,\cdot)
 \frac{b(\cdot-\varepsilon y,\cdot)-b}{\varepsilon} +zDb\,y
 \right\|_{L^1(Q_T)}\\
 &\quad\leq
 \|z\|_{L^2(Q_T)}
 \left\|
 \frac{b(\cdot-\varepsilon y,\cdot)-b}{\varepsilon}
       +Db\,y
 \right\|_{L^2(Q_T)}
 +|y|\|z(\cdot-\varepsilon y,\cdot)-z\|_{L^2(Q_T)}\|Db\|_{L^2(Q_T)}
 \longrightarrow0.
\end{align*}
The expression on the left is bounded by
$2|y|\|z\|_{L^2(Q_T)}\|Db\|_{L^2(Q_T)}$.
Because $\chi$ is smooth and compactly supported,
dominated convergence therefore gives
\[
 \int z(x-\varepsilon y,t)
       \frac{b(x-\varepsilon y,t)-b(x,t)}{\varepsilon}
                  \cdot\nabla\chi(y)\,dy
 \longrightarrow
 -z(x,t)\int (Db(x,t)y)\cdot\nabla\chi(y)\,dy
\]
in $L^1(Q_T)$. Now, integration by parts in $y$ gives:
\[
-z(x,t)\int (Db(x,t)y)\cdot\nabla\chi(y)\,dy = z\operatorname{div}b
\]

Finally,
\[
 \|(z_\varepsilon-z)\operatorname{div}b\|_1
 \leq
 \|z_\varepsilon-z\|_2\|\operatorname{div}b\|_2
 \longrightarrow0.
\]
The two commutator terms thus converge to the same expression
and cancel. Hence
\[
 C_\varepsilon\longrightarrow0
 \quad\text{in }L^1(Q_T),
\]
which in particular proves the claimed local convergence.

The mollified equation gives locally integrable time
derivatives of $z_\varepsilon$, so the ordinary Sobolev
chain rule applies:
\begin{align*}
 \partial_t\beta(z_\varepsilon)
 ={}&
 \operatorname{div}\bigl(\beta(z_\varepsilon)b\bigr)
 +[z_\varepsilon\beta'(z_\varepsilon)
                         -\beta(z_\varepsilon)]
                           \operatorname{div}b\\
 &+\beta'(z_\varepsilon)S_\varepsilon
   +\beta'(z_\varepsilon)C_\varepsilon.
\end{align*}
All $z_\varepsilon$ lie in a fixed bounded interval.
Consequently $\beta'$ is bounded on their ranges.
Using $z_\varepsilon\to z$ locally in $L^1$,
$S_\varepsilon\to S$ in $L^1$, and
$C_\varepsilon\to0$ locally in $L^1$, we may pass to the
limit and obtain:
\[
 \partial_t\beta(z)
 =\operatorname{div}(\beta(z)b)
  +[z\beta'(z)-\beta(z)]\operatorname{div}b
  +\beta'(z)S.
\]
Taking $z=n_j^i$ and $z=\rho_j$ gives the first two assertions.

The initial terms pass in the same calculation because
$z_\varepsilon(0)=z(0)*\chi_\varepsilon\to z(0)$ in $L^1$.
Furthermore, $\beta$ is Lipschitz on the bounded range of $z$,
and $\beta(0)=0$, so
\[
 \|\beta(z(t))-\beta(z(0))\|_1
 \leq L_\beta\|z(t)-z(0)\|_1\longrightarrow0.
\]

Then for the absolute value and comparison results, choose smooth functions $\theta_\varepsilon$ with
$0\leq\theta_\varepsilon\leq1$,
$\theta_\varepsilon(r)=0$ for $r\leq0$, and
$\theta_\varepsilon(r)=1$ for $r\geq\varepsilon$.
Set
\[
 \beta_\varepsilon(r)=\int_0^r\theta_\varepsilon(s)\,ds.
\]
Then
\[
 \beta_\varepsilon(r)\to r_+,\qquad
 \beta_\varepsilon'(r)\to\mathbf1_{\{r>0\}},\qquad
 |r\beta_\varepsilon'(r)-\beta_\varepsilon(r)|
       \leq\varepsilon.
\]
Applying the earlier mollifier argument and passing to the limit gives
\[
 \partial_t z_+
 =\operatorname{div}(z_+b)+\mathbf1_{\{z>0\}}S.
\]
Applying this identity to $z$ and $-z$ and adding proves
\eqref{eq:species-difference-absolute-value}.

\end{proof}


\end{document}